\documentclass[12pt,reqno]{article}

\usepackage[usenames]{color}
\usepackage{amssymb}
\usepackage{amsmath}
\usepackage{amsthm}
\usepackage{amsfonts}
\usepackage{amscd}
\usepackage{graphicx}
\usepackage{ifthen}
\usepackage{tikz}
\usetikzlibrary{calc}

\usepackage[colorlinks=true,
linkcolor=webgreen,
filecolor=webbrown,
citecolor=webgreen]{hyperref}

\definecolor{webgreen}{rgb}{0,.5,0}
\definecolor{webbrown}{rgb}{.6,0,0}

\usepackage{fullpage}
\usepackage{float}

\newcommand{\seqnum}[1]{\href{https://oeis.org/#1}{\rm \underline{#1}}}

\allowdisplaybreaks

\begin{document}

\theoremstyle{plain}
\newtheorem{theorem}{Theorem}
\newtheorem{corollary}[theorem]{Corollary}
\newtheorem{lemma}[theorem]{Lemma}
\newtheorem{proposition}[theorem]{Proposition}

\theoremstyle{definition}
\newtheorem{definition}[theorem]{Definition}
\newtheorem{example}[theorem]{Example}
\newtheorem{conjecture}[theorem]{Conjecture}
\newtheorem{notation}[theorem]{Notation}
\newtheorem{remark}[theorem]{Remark}

\theoremstyle{remark}

\begin{center}
\vskip 1cm{\LARGE\bf
Enumerating Multi-Operator Monomials in\\
\vskip .1in
Commutative and Noncommutative Settings
}
\vskip 1cm
\large
Yu Hin (Gary) Au and Murray R. Bremner\\
Department of Mathematics and Statistics\\
University of Saskatchewan\\
Saskatoon, SK S7N 5E6\\
Canada \\
\href{mailto:au@math.usask.ca}{\tt au@math.usask.ca} \\
\href{mailto:bremner@math.usask.ca}{\tt bremner@math.usask.ca} \\
\ \\
\end{center}

\begin{abstract}
We study enumeration problems for multi-operator monomials generated from one indeterminate by an associative multiplication together with finitely many unary operators. We consider four regimes, according to whether multiplication is commutative and whether the unary operators commute. In the case where the unary operators do not commute, we obtain explicit multigraded generating functions and coefficient formulas, including a multinomial refinement of the Narayana numbers, together with interpretations in terms of rooted trees, restricted lattice paths, and binary trees. When the unary operators commute, we derive canonical representatives and effective recurrences, with corresponding monotonicity conditions in the combinatorial models. When multiplication is commutative, the sequence decomposition is replaced by a multiset decomposition, leading to exp--log generating functions and Euler-transform recurrences. In special cases, the resulting sequences recover classical families including the Catalan numbers, the small Schr\"oder numbers, and rooted-tree numbers.
\end{abstract}

\section{Introduction}\label{sec1} 

Let $d\ge 1$, and let $P_1,\ldots,P_d$ be unary operators. 
We consider the collection $V_d$ of multi-operator monomials generated from a single 
indeterminate $\ast$ by repeated use of an associative binary multiplication 
(initially noncommutative) and the unary operators $P_1,\ldots,P_d$. 
More concretely, $V_d$ is the smallest set such that 
\begin{itemize} 
\item (Single indeterminate) $\ast\in V_d$;
\item (Multiplication / concatenation) if $v,v'\in V_d$, then $vv'\in V_d$;
\item (Unary operator) if $v\in V_d$, then $P_i(v)\in V_d$ for every $1\le i\le d$.
\end{itemize} 
Thus, $V_d$ is the set of monomials in the free operated semigroup on one 
generator with unary operators $P_1,\ldots,P_d$. We study four closely related enumerative regimes arising from this construction, depending on whether the unary operators commute and whether the associative multiplication is commutative.

\subsection{Enumeration problems}

We study several natural enumeration problems associated with $V_d$.
\begin{enumerate}
\item
\textbf{Degree--multiplicity enumeration.}
Given $r\ge 1$ and $\mathbf{s}=(s_1,\ldots,s_d)\in\mathbb{Z}_{\ge0}^d$, let 
\[ 
a_d(r;\mathbf{s}) 
\] 
denote the number of monomials $v\in V_d$ with exactly $r$ occurrences of the 
indeterminate $\ast$ (called the \emph{degree} of $v$) and exactly $s_i$ occurrences of the operator $P_i$ for each $1\le i\le d$ (called the \emph{multiplicity} of $P_i$ in $v$). This gives an intrinsic degree--multiplicity grading on $V_d$. For example, Figure~\ref{fig:a221boxes} lists the 30 monomials counted by $a_{2}(2; (2,1))$.

\begin{figure}[ht]
\centering
\small
\setlength{\fboxsep}{5pt}

\newcommand{\mybox}[1]{\fbox{\parbox[c][1.5cm][c]{0.35\textwidth}{\centering $#1$}}}

\begin{tabular}{cc}
\mybox{
\begin{array}{c}
P_1(P_1(P_2(\ast\ast)))\\
P_1(P_2(P_1(\ast\ast)))\\
P_2(P_1(P_1(\ast\ast)))
\end{array}}
&
\mybox{
\begin{array}{cc}
P_1(P_1(P_2(\ast)\ast)) & P_1(P_1(\ast P_2(\ast)))
\end{array}}
\\[1.2em]

\mybox{
\begin{array}{cc}
P_1(P_2(P_1(\ast)\ast)) & P_1(P_2(\ast P_1(\ast)))\\
P_2(P_1(P_1(\ast)\ast)) & P_2(P_1(\ast P_1(\ast)))
\end{array}}
&
\mybox{
\begin{array}{cc}
P_1(P_1(P_2(\ast))\ast) & P_1(\ast P_1(P_2(\ast)))\\
P_1(P_2(P_1(\ast))\ast) & P_1(\ast P_2(P_1(\ast)))
\end{array}}
\\[1.2em]

\mybox{
\begin{array}{cc}
P_2(P_1(P_1(\ast))\ast) & P_2(\ast P_1(P_1(\ast)))
\end{array}}
&
\mybox{
\begin{array}{cc}
P_1(P_1(\ast)P_2(\ast)) & P_1(P_2(\ast)P_1(\ast))
\end{array}}
\\[1.2em]

\mybox{
\begin{array}{c}
P_2(P_1(\ast)P_1(\ast))
\end{array}}
&
\mybox{
\begin{array}{cc}
P_1(P_1(P_2(\ast)))\ast & \ast P_1(P_1(P_2(\ast)))\\
P_1(P_2(P_1(\ast)))\ast & \ast P_1(P_2(P_1(\ast)))\\
P_2(P_1(P_1(\ast)))\ast & \ast P_2(P_1(P_1(\ast)))
\end{array}}
\\[1.2em]

\mybox{
\begin{array}{cc}
P_1(P_1(\ast))P_2(\ast) & P_2(\ast)P_1(P_1(\ast))
\end{array}}
&
\mybox{
\begin{array}{cc}
P_1(P_2(\ast))P_1(\ast) & P_1(\ast)P_1(P_2(\ast))\\
P_2(P_1(\ast))P_1(\ast) & P_1(\ast)P_2(P_1(\ast))
\end{array}}
\end{tabular}

\caption{The 30 monomials counted by $a_2(2;(2,1))$, arranged in 10 boxes.
Within each box, monomials in the same row become equivalent when multiplication is
commutative, and monomials in the same column become equivalent when $P_1$ and $P_2$ commute.}\label{fig:a221boxes}
\end{figure}

\item
\textbf{Length-graded enumeration via bracketed words.}
A second family of problems arises by encoding each monomial in $V_d$ as a bracketed 
word. For example, if $d=2$ and we use the delimiters $(\cdot)$ and $[\cdot]$ 
to denote the unary operators $P_1$ and $P_2$ respectively, then the monomial
${\ast}P_1({\ast}P_2({\ast}{\ast}))$
is encoded by the bracketed word
${\ast}({\ast}[{\ast}{\ast}])$.
We assign length 1 to each opening and closing delimiter, and length $\ell$ to the symbol $\ast$, where $\ell \ge 1$ is fixed. Thus, the bracketed word displayed above has length $4\ell + 4$. For $\ell = 1$ and $\ell=2$, we obtain a natural correspondence between these bracketed words and Motzkin and Schr\"{o}der paths, respectively. 
This monomial-to-word correspondence leads to length-graded enumeration problems whose resulting sequences generalize the classical Catalan numbers and the small Schr\"{o}der numbers (respectively \seqnum{A000108} and \seqnum{A001003} in the On-Line Encyclopedia of Integer Sequences (OEIS)~\cite{OEIS}).

\item
\textbf{Enumeration modulo commutativity of the unary operators.}
A third family of problems arises when we impose the additional relations 
\[
P_i(P_j(v))\sim P_j(P_i(v))
\qquad (1\le i < j\le d,\; v\in V_d),
\]
so that the unary operators commute pairwise. For example, in Figure~\ref{fig:a221boxes}, monomials in the same column within the same box are equivalent when $P_1$ and $P_2$ commute. One may then count equivalence classes of monomials under the congruence generated by 
these relations, or equivalently work with canonical representatives in which the unary 
operators are arranged in a fixed order along each maximal unary chain. 
This leads to commutative analogs of the multigraded counts and of the length-graded 
sequences above. 

\item
\textbf{Enumeration modulo commutativity of the binary operation.}
Finally, one may also ask what happens when the associative binary multiplication is made
commutative. For example, in Figure~\ref{fig:a221boxes}, monomials in the same row within the same box are equivalent when multiplication is commutative.  We show that the resulting generating functions are governed by the classical
Euler transform, and that in this case we recover the rooted-tree sequence
\seqnum{A000081}, together with a number of generalizations.
\end{enumerate}

These combinatorial problems originated in Bremner's work on nonassociative algebras and algebraic operads. For an introduction to operads with an emphasis on algorithmic aspects, see Bremner and Dotsenko~\cite{BD}. In particular, consider the nonsymmetric operad generated by a noncommutative associative (binary) 
operation and a finite number $d$ of commuting unary operators;
this operad is multigraded by the degree $r$ of the indeterminate $\ast$ and the multiplicities 
$\mathbf{s}=(s_1,\ldots,s_d)$ of the operators.
One of the main results of this paper (Theorem~\ref{thm:sec3multirec}) gives an effective
recurrence for the dimensions $a^c_d(r;\mathbf{s})$ of the homogeneous components of this operad.  
An important quotient of this operad occurs when each unary operator satisfies the Rota-Baxter (RB) relation; 
see Rota \cite{Rota}. 
Aguiar and Moreira \cite{AM} discuss the combinatorics of free RB algebras.  
Guo and Sit \cite{GuoSit} study the enumeration of RB words.  
Aguiar \cite{Aguiar} showed that from an RB operator on an associative algebra one may construct 
a dendriform algebra; for the latter, see Loday \cite[\S 5]{Loday}.  
From two commuting RB operators one obtains the quadri-algebras of Aguiar and Loday \cite{AL}.  
Three commuting RB operators produce the octo-algebras of Leroux \cite[\S 5]{Leroux}.  
Extension of this construction to an arbitrary finite number of commuting RB operators has been studied 
by Ebrahimi-Fard and Guo \cite[Corollary 5.3]{EFG}.  
Gr\"obner bases for multi-operator algebras have been constructed by Bokut et al.~\cite{BCQ} 
and Liu et al.~\cite{LQQZ}.

Our present work also builds on the well-developed literature on combinatorial realizations 
of algebraic objects. 
In particular, Guo \cite{GuoOSMRT} treated operated semigroups via Motzkin paths and rooted trees, 
and this was followed by work on free operated monoids and on bracketed subwords in 
free operated semigroups \cite{GuoZhengRLS,ZhangGaoFOM}. 
On the enumerative side, generalized Motzkin and Schr\"oder paths with colored up-steps or horizontal steps of fixed length have been studied from several perspectives; see, for example, Sulanke \cite{SulankeGM}, Stein and Waterman \cite{SteinWaterman}, and later work on hill-free, colored, and higher-order Motzkin-type families, as well as relations between Schr\"oder and Motzkin paths \cite{BarcucciPergolaPinzaniRinaldi,BrennanKnopfmacher,DeJagerNaquinSeidlDrube,
HaugPrellbergSiudem,MansourSchorkSun,YangYang}.

While these free operated constructions and their realizations are by now standard, our focus here is explicit enumeration in the one-generator setting with several unary operators across all four commutativity regimes described above. In the noncommutative-multiplication settings, we obtain multigraded generating functions, closed formulas, and effective recurrences, together with path and tree models. When multiplication is commutative, the sequence-of-atoms decomposition is replaced by a multiset-of-atoms decomposition, leading to Euler-transform recurrences and rooted-tree
analogs. In the one-operator case we recover several classical sequences, while for $d\ge 2$ many of the resulting families appear not to have been recorded previously in the OEIS. As part of this project, we submitted 15 new sequences to the OEIS under the allocated identifiers \seqnum{A394939}--\seqnum{A394953}. For OEIS submission purposes, we include the initial empty object when appropriate, so the corresponding entries naturally use offset 0. Throughout the paper, however, we index the length-graded families from the first positive size, since this convention is more convenient for our statements and tables.

For convenience, Table~\ref{tab:regimes} summarizes the four commutativity regimes
studied in this paper, together with the corresponding decomposition principles,
main enumerative outputs, and the sections in which they are treated.

\begin{table}[ht]
\scriptsize
\centering
\begin{tabular}{|l|l|p{3cm}|p{5cm}|l|}
\hline
Unary operators & Multiplication & Basic decomposition & Main outputs & Section \\
\hline
Noncommuting & Noncommutative & Sequence of atoms &
Explicit multigraded generating functions; closed coefficient formulas; multinomial refinement of Narayana numbers &
Section~\ref{sec2} \\
\hline
Commuting & Noncommutative & Sequence of atoms with canonical unary ordering &
Multigraded functional equations and effective recurrences; canonical representatives with monotonicity conditions &
Section~\ref{sec3} \\
\hline
Noncommuting & Commutative & Multiset of atoms &
Exp--log generating functions; Euler-transform recurrences &
Section~\ref{sec41} \\
\hline
Commuting & Commutative & Multiset of atoms with canonical unary ordering &
Euler-transform recurrences with commuting-unary refinements &
Section~\ref{sec42} \\
\hline
\end{tabular}
\caption{Summary of the four commutativity regimes studied in this paper.}
\label{tab:regimes}
\end{table}

\subsection{A roadmap of the paper}

Section~\ref{sec2} studies the noncommuting unary case, in which the associative multiplication is noncommutative and the unary operators do not commute. We derive the multigraded generating function and an explicit coefficient formula for $a_d(r;\mathbf{s})$, then specialize to the length-graded families $b_{d,\ell}(n)$ and describe several combinatorial interpretations.

Section~\ref{sec3} turns to the case in which the unary operators commute while the binary operation remains noncommutative. We introduce canonical representatives for the corresponding equivalence classes, derive the multigraded functional equation and recurrences, and then study the associated length-graded sequences and their combinatorial models.

Section~\ref{sec4} considers the extension in which the associative binary multiplication is commutative. The sequence-of-atoms decomposition from Sections~\ref{sec2} and~\ref{sec3} is then replaced by a multiset-of-atoms decomposition, leading to Euler-transform-type generating functions and recurrences in both the noncommuting-unary and commuting-unary settings. In the case $d=1$, this recovers several classical rooted-tree sequences. Section~\ref{sec5} then compares the asymptotic behavior of the four families of length-graded sequences studied in the paper.

\section{The noncommuting unary case}\label{sec2}

In this section we assume that the binary multiplication is associative but not commutative,
and that the unary operators $P_1,\ldots,P_d$
are pairwise distinct and do \emph{not} commute with one another.
Thus, two monomials that differ only by permuting nested unary operators are regarded as distinct. We study the resulting enumeration problems in three steps:
first by degree and operator multiplicity, then by two natural length gradings, and
finally via several combinatorial interpretations.

\subsection{Degree and multiplicity grading}

We begin with the most refined enumeration problem. Given an integer $k \geq 1$, let
\[
[k] := \{ 1, 2, \ldots, k\}.
\]
For $r\ge 1$ and $\mathbf{s}=(s_1,\ldots,s_d)\in\mathbb{Z}_{\ge0}^d$, recall that $a_d(r;\mathbf{s})$ denotes the number of monomials in $V_d$ containing exactly $r$
occurrences of the indeterminate $\ast$ and exactly $s_i$ occurrences of the unary
operator $P_i$ for each $i \in [d]$. We encode these numbers in the multivariate generating function
\[
A_d(u;\mathbf{p})
:=
\sum_{r\ge1}\;\sum_{s_1,\ldots,s_d\ge0}
a_d(r;\mathbf{s})\,u^r p_1^{s_1}\cdots p_d^{s_d},
\]
where $\mathbf{p}=(p_1,\ldots,p_d)$. Next, we say that a monomial is an \emph{atom} if it cannot be written as a product of
two nontrivial monomials.
Equivalently, an atom is either the single indeterminate $\ast$, or a monomial of the
form $P_i(v)$ where $i \in [d]$ and $v\in V_d$.
Let
\[
\overline a_d(r;\mathbf{s})
\]
denote the corresponding multigraded atom count, and let
\[
\overline A_d(u; \mathbf{p})
:=
\sum_{r\ge1}\;\sum_{s_1,\ldots,s_d\ge0}
\overline a_d(r;\mathbf{s})\,u^r p_1^{s_1}\cdots p_d^{s_d}
\]
be the atom generating function.

This atom decomposition gives two immediate relations between $A_d$ and $\overline A_d$.
First, every monomial factors uniquely as a sequence of atoms.
Thus, if we let $\overline V_d\subseteq V_d$ be the set of atoms, then there is a natural
bijection between $V_d$ and
\[
\overline V_d\;\cup\;\overline V_d^2\;\cup\;\overline V_d^3\;\cup\;\cdots.
\]
Consequently,
\begin{equation}\label{eq:sec2seq}
A_d(u;\mathbf{p})
=
\sum_{j \ge1}\bigl(\,\overline A_d(u;\mathbf{p})\bigr)^j
=
\frac{\overline A_d(u;\mathbf{p})}{1-\overline A_d(u;\mathbf{p})}.
\end{equation}
Second, an atom is either the single indeterminate $\ast$, contributing $u$, or the result
of applying one of the operators $P_1,\ldots,P_d$ to an arbitrary monomial, contributing
\[
(p_1+\cdots+p_d)\,A_d(u;\mathbf{p}).
\]
Hence
\begin{equation}\label{eq:sec2atom}
\overline A_d(u;\mathbf{p})
=u+(p_1+\cdots+p_d)\,A_d(u;\mathbf{p}).
\end{equation}
Using the notation
\[
|\mathbf{p}| :=p_1+\cdots+p_d,
\]
as well as combining~\eqref{eq:sec2seq} and~\eqref{eq:sec2atom}, we obtain the quadratic equation
\begin{equation}\label{eq:sec2masterquad}
|\mathbf{p}| A_d(u;\mathbf{p})^2-(1-u- |\mathbf{p}|)A_d(u;\mathbf{p})+u=0.
\end{equation}
Then we have the following.

\begin{theorem}\label{thm:sec2multigradedgf}
Let $|\mathbf{p}| =p_1+\cdots+p_d$. Then
\begin{equation}\label{eq:sec2Aclosed}
A_d(u;\mathbf{p})
=
\frac{1-u-|\mathbf{p}|-\sqrt{(1-u-|\mathbf{p}|)^2-4u|\mathbf{p}|}}{2|\mathbf{p}|}.
\end{equation}
\end{theorem}

\begin{proof}
We solve \eqref{eq:sec2masterquad} for $A_d(u;\mathbf{p})$;
the branch satisfying $A_d(0;\mathbf{p})=0$ is~\eqref{eq:sec2Aclosed}.
\end{proof}

For future reference, we note that rearranging \eqref{eq:sec2masterquad} gives
\begin{equation}\label{eq:sec2masterfunc}
A_d(u;\mathbf{p})
=u+uA_d(u;\mathbf{p})+|\mathbf{p}| A_d(u;\mathbf{p})+|\mathbf{p}|A_d(u;\mathbf{p})^2.
\end{equation}

We next provide a closed-form formula for $a_d(r; \mathbf{s})$. We begin with the known one-operator case $d=1$. Given integers $n,k$ where $0 \leq k \leq n-1$, let 
\[
N(n,k) := \frac{1}{n} \binom{n}{k} \binom{n}{k+1}
\]
denote the \emph{Narayana numbers}~\seqnum{A001263}. Then Bremner and Elgendy~\cite[Lemma 2.5]{BremnerE22} showed that
\begin{equation}\label{eq:sec2multigradedd=1}
a_1(r;k) = N(r+k, k)
\end{equation}
for every $r \geq 1$ and $k \geq 0$. 
A simple combinatorial proof of this is obtained by mapping each monomial counted by $a_1(r;k)$ 
(in bracketed-word format) to Dyck paths as follows:
\begin{itemize}
\item
each opening bracket maps to an up step $U$;
\item
each closing bracket maps to a down step $D$;
\item
each instance of the indeterminate $\ast$ maps to a peak $UD$.
\end{itemize}
This gives a bijection between the monomials counted by $a_1(r;k)$ and Dyck paths with semilength $r+k$ and $r$ peaks, which is counted by $N(r+k,k)$.  Subsequently, Au and Bremner~\cite{AuB25} further generalized this result to count operator monomials generated by one unary operator and a multiplication operation of arbitrary arity, obtaining a family of generalized Narayana numbers. For introductions to the Narayana numbers, see Petersen~\cite[Chapter 2]{Petersen15} or Grimaldi~\cite[Chapter 33]{Grimaldi12}.

Next, we show that $a_d(r;\mathbf{s})$ is simply the multinomial refinement of the
one-operator formula~\eqref{eq:sec2multigradedd=1}.

\begin{theorem}\label{thm:sec2multigradedformula}
Let $d\ge 1$ and fix $\mathbf{s}=(s_1,\ldots,s_d)\in\mathbb{Z}_{\ge0}^d$.
Then
\begin{equation}\label{eq:sec2multigradedformula}
a_d(r;\mathbf{s})
=
\binom{|\mathbf{s}|}{s_1,\ldots,s_d}\,
N(r+|\mathbf{s}|,|\mathbf{s}|).
\end{equation}
\end{theorem}

\begin{proof}
Given a monomial in $V_d$, identify the $d$ operators $P_1,\ldots,P_d$ with a single operator $P$. This maps every monomial counted by $a_d(r;\mathbf{s})$ to a monomial counted by $a_1(r;|\mathbf{s}|)$. Conversely, given any such one-operator monomial, one recovers a monomial with
multiplicity vector $\mathbf{s}$ by arbitrarily recoloring the $|\mathbf{s}|$ occurrences of $P$ so that
exactly $s_i$ of them receive color $i$.
This can be done in
$\binom{|\mathbf{s}|}{s_1,\ldots,s_d}$
ways, which proves~\eqref{eq:sec2multigradedformula}.
\end{proof}

\subsection{Length-graded specializations}

We now consider a family of one-variable specializations determined by the assigned length of the indeterminate $\ast$. Let $\ell \geq 1$ be a fixed integer, and suppose we encode monomials in $V_d$ as bracketed words where each delimiter (left or right) has length 1, and the indeterminate symbol $\ast$ has length $\ell$. If a monomial has degree $r$ and multiplicity vector $\mathbf{s}$
then its total length is $n=\ell r+2|\mathbf{s}|$. Accordingly, for every $\ell\ge 1$ and $n\ge 1$, we define
\begin{equation}
\label{murray1}
b_{d,\ell}(n)
:=
\sum_{\substack{r\ge1,\;\mathbf{s}\ge\mathbf{0}\\ \ell r+2|\mathbf{s}|=n}}
a_d(r;\mathbf{s}),
\end{equation}
where $\mathbf{s}\ge\mathbf{0}$ is interpreted coordinate-wise.
Equivalently, if
\[
B_{d,\ell}(z):=\sum_{n\ge1} b_{d,\ell}(n)z^n,
\]
then
\begin{equation}
\label{murray2}
B_{d,\ell}(z)=A_d(z^\ell;z^2,\ldots,z^2).
\end{equation}
Recall that 
\[
N(n,k)=\frac{1}{n}\binom{n}{k}\binom{n}{k+1}
\]
denotes the (shifted) Narayana number. 
We show that, for all $\ell \geq 1$, the sequences $b_{d,\ell}(n)$ admit a compact description 
in terms of $N(n,k)$.

\begin{theorem}\label{thm:sec2bdlgeneral}
Let $\ell\ge 1$ be an integer.
Then the generating function $B_{d,\ell}(z)$ satisfies
\begin{equation}\label{eq:sec2bdlfunc}
B_{d,\ell}(z)
=
z^\ell+z^\ell B_{d,\ell}(z)+dz^2B_{d,\ell}(z)+dz^2B_{d,\ell}(z)^2.
\end{equation}
Equivalently,
\begin{equation}\label{eq:sec2bdlquad}
dz^2B_{d,\ell}(z)^2-(1-z^\ell-dz^2)B_{d,\ell}(z)+z^\ell=0.
\end{equation}
Moreover, for every $n\ge 1$,
\begin{equation}\label{eq:sec2bdlnarayanak}
b_{d,\ell}(n)
=
\sum_{\substack{0 \leq k \leq (n-\ell)/2 \\ n-2k\equiv 0\!\!\!\!\pmod{\ell}}}
d^k
N\!\left(\frac{n-2k}{\ell}+k,\,k\right).
\end{equation}
\end{theorem}

\begin{proof}
The generating function identities follow immediately from the specialization \eqref{murray2}
together with the functional equation~\eqref{eq:sec2masterfunc} for $A_d$.

To obtain the coefficient formula, observe that by definition \eqref{murray1} we have
\[
b_{d,\ell}(n)
=
\sum_{\substack{r\ge1,\;k\ge0\\ \ell r+2k=n}}
\ \sum_{|\mathbf{s}|=k}
a_d(r;\mathbf{s}).
\]
Summing the multinomial coefficients over all $\mathbf{s}$ with
$|\mathbf{s}|=k$, we obtain
\[
\sum_{|\mathbf{s}|=k}\binom{k}{s_1,\ldots,s_d}=d^k.
\]
Using Theorem~\ref{thm:sec2multigradedformula}, this gives
\begin{equation}\label{eq:sec2bdlnarayanark}
b_{d,\ell}(n)
=
\sum_{\substack{r\ge1,\;k\ge0\\ \ell r+2k=n}}
d^k\,N(r+k,k).
\end{equation}
Finally, substituting $r=(n-2k)/\ell$
into~\eqref{eq:sec2bdlnarayanark} yields~\eqref{eq:sec2bdlnarayanak}.
\end{proof}

Theorem~\ref{thm:sec2bdlgeneral} shows that, in the noncommuting-unary case, the dependence on the $d$ unary operators enters only through the weight $d^k$ assigned to the total number of unary-operator applications. Equivalently, the multigraded count is a multinomial refinement of the one-operator Narayana formula. Furthermore, observe that when $\ell$ is even, $b_{d,\ell}(n) = 0$ for all odd $n$. 

The cases $\ell=1$ and $\ell=2$ are particularly noteworthy for the combinatorial interpretations below:

\begin{corollary}\label{cor:sec2bdlspecial}
For every $n \geq 1$, we have
\begin{align}
\label{eq:sec2bd1narayana}
b_{d,1}(n)
&=
\sum_{k=0}^{\lfloor (n-1)/2\rfloor}
d^k\,N(n-k,k), \\
\label{eq:sec2bd2narayana}
b_{d,2}(2n) &=
\sum_{k=0}^{n-1} d^k\,N(n,k).
\end{align}
\end{corollary}

\begin{proof}
For $\ell=1$, equation~\eqref{eq:sec2bdlnarayanak} reduces immediately to
\eqref{eq:sec2bd1narayana}, since then
$r=n-2k$ and $r+k=n-k$.
For $\ell=2$, write the total length as $2n$.
Then $2r+2k=2n \iff r+k=n$
and~\eqref{eq:sec2bdlnarayanark} becomes exactly
\eqref{eq:sec2bd2narayana}.
\end{proof}

\begin{table}[ht]
\centering
\begin{tabular}{c|l|l}
$d$ & First terms of $b_{d,1}(n)$ & OEIS entry \\
\hline
\\[-12pt]
1 & $1, 1, 2, 4, 8, 17, 37, 82, 185, 423, 978, 2283, 5373, 12735, \ldots$ & \seqnum{A004148} \\
2 & $1,1,3,7,17,45,119,323,893,2497,7067,20191,58153,\ldots$ & \seqnum{A394939}\textsuperscript{*} \\
3 & $1,1,4,10,28,85,253,784,2461,7813,25138,81571,\ldots$ & \seqnum{A394940}\textsuperscript{*} \\
4 & $1,1,5,13,41,137,445,1525,5249,18321,64821,231069,\ldots$ & \seqnum{A394941}\textsuperscript{*}
\end{tabular}
\caption{First terms of the sequences $b_{d,1}(n)$ for $d=1,2,3,4$.}
\label{tab:bd1}
\end{table}

Table~\ref{tab:bd1} records the first few terms of $b_{d,1}(n)$. The case $d=1$ agrees with \seqnum{A004148}. Sequences marked with an asterisk were submitted to the OEIS as part of this project.

\begin{table}[ht]
\centering
\begin{tabular}{c|l|l}
$d$ & First terms of $b_{d,2}(2n)$ & OEIS entry \\
\hline
\\[-12pt]
1 & $1,2,5,14,42,132,429,1430,4862,16796,58786,208012,742900,\ldots$ & \seqnum{A000108} \\
2 & $1,3,11,45,197,903,4279,20793,103049,518859,2646723,\ldots$ & \seqnum{A001003} \\
3 & $1,4,19,100,562,3304,20071,124996,793774,5120632,33463102,\ldots$ & \seqnum{A007564} \\
4 & $1, 5, 29, 185, 1257, 8925, 65445, 491825, 3768209, 29324405,\ldots$ & \seqnum{A059231}
\end{tabular}
\caption{First terms of the sequences $b_{d,2}(2n)$ for $d=1,2,3,4$.}
\label{tab:bd2}
\end{table}

Table~\ref{tab:bd2} shows that the family $b_{d,2}(2n)$ begins with
two classical sequences:
$d=1$ gives the Catalan numbers \seqnum{A000108}, 
and $d=2$ gives the small Schr\"oder numbers \seqnum{A001003}.

As for $\ell \geq 3$, the initial terms of $b_{1,3}(n)$ are
\[
0,0,1,0,1,1,1,3,2,6,7,11,21,25,52,71,121,204,302,547,828,\ldots
\]
which is \seqnum{A329691} with the first two terms removed. Interestingly, $b_{1,4}(2n)$ begins with 
\[
0,1,1,2,4,8,17,37,82,185,423, 978, 2283, 5373, 12735, 30372, \ldots,
\]
which agrees with $b_{1,1}(n)$ after the initial zero term. In fact, one can apply \eqref{eq:sec2bdlfunc} and show that
$B_{1,4}(z^2) = z^2 B_{1,1}(z)$,
which implies that
$b_{1,4}(2n+2) = b_{1,1}(n)$
for every $n \geq 1$. Other $b_{d,\ell}(n)$ sequences cataloged in the OEIS at the time of this writing include 
$b_{2,4}(2n)$ (\seqnum{A159771}), 
$b_{4,4}(2n)$ (\seqnum{A239204}), 
$b_{1,6}(2n)$ (\seqnum{A023432}), 
$b_{1,8}(2n)$ (\seqnum{A023427}), 
and
$b_{1,10}(2n)$ (\seqnum{A212364}).

\subsection{Combinatorial interpretations}

We now describe several combinatorial models for the noncommuting unary case.
These models explain both the multigraded counts $a_d(r;\mathbf{s})$ and the
length-graded sequences $b_{d,\ell}(n)$.
The tree model most directly reflects the recursive construction of operator monomials,
while the path model explains the length gradings.
In the special case $\ell=2$, the same objects also admit a particularly natural
interpretation in terms of rooted ordered binary trees.

\subsubsection{Rooted ordered trees}

Every monomial in $V_d$ may be represented by a rooted ordered tree as
follows:
A leaf represents the indeterminate $\ast$, a unary internal node labeled $i\in[d]$ represents the application of $P_i$, and a product $v_1 v_2 \cdots v_m$ ($m \ge 2$)
is represented by an unlabeled internal node with $m$ ordered children corresponding to
the associative product of $v_1,\ldots,v_m$.
Because the multiplication is associative but not commutative, the left-to-right order
of the children matters, so one obtains a rooted ordered tree.

\begin{proposition}\label{prop:sec23trees}
For integers $r\ge1$ and $\mathbf{s}=(s_1,\ldots,s_d)\in\mathbb{Z}_{\ge0}^d$,
the number $a_d(r;\mathbf{s})$ counts rooted ordered trees with exactly $r$ leaves and
exactly $s_i$ unary internal nodes labeled $i$, for each $1\le i\le d$, in which every
non-unary internal node has at least two children.
\end{proposition}

\begin{proof}
The correspondence described above is recursive and invertible.
Each occurrence of the indeterminate $\ast$ contributes one leaf, and each occurrence of
the operator $P_i$ contributes one unary internal node labeled $i$.
Conversely, every such rooted ordered tree determines a unique monomial in $V_d$ by reading the tree from the leaves upward.
\end{proof}

Under this correspondence, the atoms are exactly the trees consisting either of a single
leaf or of a unary root whose unique child is an arbitrary tree.
Thus, the rooted ordered tree model gives a direct combinatorial realization of the defining
recursion for $V_d$.

\subsubsection{Restricted Motzkin, Sch\"{o}der, and Dyck paths}

We next describe a path model that simultaneously explains the length-graded families
$b_{d,\ell}(n)$ for all $\ell\ge1$.
Fix $\ell\ge1$.
To a bracketed-word realization of a monomial, we associate a lattice path by the
following rules:
\begin{itemize}
\item each opening bracket of type $i$ is mapped to an up-step $U_i=(1,1)$ labeled by $i$;
\item each closing bracket is mapped to a down-step $D=(1,-1)$;
\item each occurrence of the indeterminate $\ast$ is mapped to a horizontal step $H_\ell=(\ell,0)$.
\end{itemize}

Since every opening bracket must eventually be matched by a closing bracket, and since
properly bracketed words never contain more closing brackets than opening brackets in any
prefix, the resulting path remains on or above the $x$-axis and returns to the $x$-axis
at the end.
Moreover, the pattern $UD$ never occurs: an up-step corresponds to an opening bracket,
and the interior of the corresponding bracket pair must contain at least one symbol, so an
immediate down-step is impossible.

\begin{proposition}\label{prop:sec23pathsgeneral}
Let $\ell\ge1$ and $n\ge1$.
Then $b_{d,\ell}(n)$ counts lattice paths of total horizontal length $n$ that:
\begin{itemize}
\item start at $(0,0)$ and end on the $x$-axis;
\item use only steps $U_i=(1,1)$ for $1\le i\le d$, $D=(1,-1)$, and
$H_\ell=(\ell,0)$;
\item remain on or above the $x$-axis;
\item avoid peaks, that is, do not contain an instance of $UD$.
\end{itemize}
\end{proposition}

\begin{proof}
The bracketed-word encoding gives a bijection between monomials counted by
$b_{d,\ell}(n)$ and paths of the stated type.
The total length of the word is exactly the total horizontal length of the path.
The prefix condition for well-formed bracketed words guarantees that the path never falls
below the $x$-axis, and the fact that every opening bracket must be matched implies that the
path ends on the $x$-axis.
Finally, an instance of $UD$ would correspond to an opening bracket immediately followed by
its matching closing bracket, which is impossible because the interior of a unary operator
must contain at least one symbol.
\end{proof}

The two most noteworthy cases are as follows.

\begin{corollary}\label{cor:sec23pathsspecial}
For every $n \geq 1$,
\begin{enumerate}
\item[\textup{(i)}]
$b_{d,1}(n)$ counts peakless Motzkin paths of length $n$ with $d$ types of up-steps.
\item[\textup{(ii)}]
$b_{d,2}(2n)$ counts peakless Schr\"oder paths of semilength $n$ with $d$ types of
up-steps.
\end{enumerate}
\end{corollary}

\begin{proof}
If $\ell=1$, then the horizontal step is $H_1=(1,0)$, so the paths are precisely Motzkin
paths of length $n$ with colored up-steps and no peaks.
If $\ell=2$, then the horizontal step is $H_2=(2,0)$, so the paths are Schr\"oder paths.
Since the total horizontal length is $2n$, the natural size parameter is the semilength
$n$.
\end{proof}

For instance, we saw in the preceding section that $b_{2,2}(2n)$ gives the small Schr\"{o}der numbers (\seqnum{A001003}). Thus, Corollary~\ref{cor:sec23pathsspecial} gives a combinatorial interpretation of this classic sequence in terms of peakless, bi-colored Schr\"{o}der paths.

Next, we show that the one-operator case admits a further Dyck-path interpretation in the even-length specialization, as pointed out in several OEIS entries (such as \seqnum{A023427} and \seqnum{A212364}).

\begin{proposition}\label{prop:sec23dyckmodl}
For every integer $\ell \ge 1$ and every $n \ge \ell$, the number $b_{1,2\ell}(2n)$
counts Dyck paths of semilength $n-\ell+1$ in which every ascent and every descent has
length congruent to $1 \pmod{\ell}$.
\end{proposition}

\begin{proof}
By~\eqref{murray2},
\[
B_{1,2\ell}(z)=A_1(z^{2\ell};z^2).
\]
On the other hand, Theorem~\ref{thm:sec2multigradedgf} gives
\[
A_1(u;p)
=
\frac{1-u-p-\sqrt{(1-u-p)^2-4up}}{2p},
\]
from which one immediately checks that
\[
A_1(u;p)=\frac{u}{p}\,A_1(p;u).
\]
Substituting $u = z^{2\ell}$ and $p = z^2$ gives
\[
A_1(z^{2\ell};z^2)=z^{2\ell-2}A_1(z^2;z^{2\ell}),
\]
and therefore
\[
b_{1,2\ell}(2n)
=
[z^{2n}]B_{1,2\ell}(z)
=
[z^{2n-2\ell+2}]A_1(z^2;z^{2\ell}).
\]
Setting $x=z^2$, this becomes
\[
b_{1,2\ell}(2n)=[x^{\,n-\ell+1}]A_1(x;x^\ell).
\]

It remains to interpret the coefficients of $A_1(x;x^\ell)$.
Given a bracketed word representing a one-operator monomial, replace each opening
bracket by $U^\ell$, each occurrence of $\ast$ by $UD$, and each closing bracket by
$D^\ell$. If the monomial has degree $r$ and unary multiplicity $k$, then the resulting
Dyck path has semilength $r+\ell k$, so this construction is enumerated by
$A_1(x;x^\ell)$.

Moreover, every maximal ascent in the image consists of some number of blocks $U^\ell$
followed by the single up-step coming from an occurrence of $\ast$, and hence has length
of the form $\ell a+1$. Similarly, every maximal descent has length of the form
$\ell b+1$. Thus every ascent and every descent has length congruent to $1 \pmod{\ell}$.

Conversely, let $P$ be a Dyck path in which every ascent and every descent has length
congruent to $1 \pmod{\ell}$. Then $P$ can be uniquely decomposed into blocks of $U^{\ell}$, $D^{\ell}$, and $UD$, and we can reverse the mapping to obtain a well-formed bracketed word. 

Therefore $[x^m]A_1(x;x^\ell)$ counts Dyck paths of semilength $m$ in which every ascent
and every descent has length congruent to $1 \pmod{\ell}$. Taking
$m=n-\ell+1$ proves the result.
\end{proof}

\subsubsection{Binary trees with labeled right edges}

Consider rooted ordered binary trees in which each left edge is unlabeled and each right edge is
labeled by an element of $[d]$. There is a natural recursive correspondence between such trees with $n$ vertices and the objects counted by $b_{d,2}(2n)$.

\begin{proposition}\label{prop:sec23binary}
For every $n\ge1$, $b_{d,2}(2n)$ counts rooted ordered binary trees with $n$ vertices
in which every right edge is assigned one of $d$ possible labels.
\end{proposition}

\begin{proof}
We define a recursive map $f$ from rooted ordered binary trees to monomials. First, $f$ maps the empty tree to the empty monomial. Next, given a nonempty tree $T$, if the root has no right child, and $T_L$ denotes its (possibly empty) left subtree, let
\[
f(T) =  f(T_L) \ast.
\]
If the root has a right child joined by an edge labeled $i$, and $T_L$, $T_R$ denote the
left and right subtrees respectively, let
\[
f(T) = f(T_L) P_i \bigl(f(T_R)\bigr).
\]
The construction is invertible, and hence defines a bijection. Moreover, if $T$ has $n$ vertices, then the bracketed-word realization of $f(T)$ (with $\ast$ assigned length 2) has length $2n$, and so our claim follows.
\end{proof}

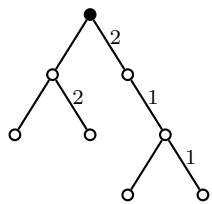
\begin{figure}[htbp]
\centering
\def\sc{0.5}
\def\ysc{0.8}

\begin{tikzpicture}[xscale=\sc, yscale=\ysc,thick,main node/.style={circle, inner sep=0.5mm, draw, font=\small\sffamily},edge label/.style={font=\scriptsize,  inner sep=0pt}]
\node[main node,fill] at (2,3) (0) {};
\node[main node] at (1,2) (1) {};
\node[main node] at (3,2) (2) {};
\node[main node] at (0,1) (11) {};
\node[main node] at (2,1) (12) {};
\node[main node] at (4,1) (22) {};
\node[main node] at (3,0) (221) {};
\node[main node] at (5,0) (222) {};

\path
(0) edge (1)
(0) edge node[pos=0.5, edge label, above right] {$2$} (2)
(1) edge (11)
(1) edge node[pos=0.5, edge label, above right] {$2$} (12)
(2) edge node[pos=0.5, edge label, above right] {$1$} (22)
(22) edge (221)
(22) edge node[pos=0.5, edge label, above right] {$1$} (222);
\end{tikzpicture}
\caption{Illustrating the tree-to-monomial mapping in the proof of Proposition~\ref{prop:sec23binary}.}
\label{fig:treetomonomial}
\end{figure}

For an example, let $T$ be the tree as shown in Figure~\ref{fig:treetomonomial}, and $T_L$, $T_R$ respectively be the left and right subtrees of $T$. Then the left subtree $T_L$ has a right edge labeled $2$ and a left leaf, so
\[
f(T_L)=\ast P_2(\ast).
\]
Similarly, the root of the right subtree $T_R$ has no left child and a right edge labeled $1$,
while its right child again has no left child and a right edge labeled $1$; hence
\[
f(T_R)=P_1\bigl(\ast P_1(\ast)\bigr).
\]
Hence, we have
\[
f(T) =f(T_L) P_2 \bigl(f(T_R)\bigr)  =  \ast P_2(\ast)  P_2 \left( P_1\bigl(\ast P_1(\ast)\bigr) \right) .
\]
Since the tree $T$ has 8 vertices, the bracketed-word realization of $f(T)$ (with $\ast$ having length 2) has length 16, in agreement with Proposition~\ref{prop:sec23binary}. Also, Proposition~\ref{prop:sec23binary} explains why the generating function for $b_{d,2}(n)$ 
(see Theorem \ref{thm:sec2bdlgeneral}) satisfies the especially-simple equation
\[
B_{d,2}(z)=z(1+B_{d,2}(z))(1+d\,B_{d,2}(z)).
\]

\section{The commuting unary case}\label{sec3}

In this section we impose the additional relations
\begin{equation}\label{eq:sec3:eqrel}
P_i(P_j(v))\sim P_j(P_i(v))
\qquad (1\le i < j\le d,\; v\in V_d),
\end{equation}
so that the unary operators commute pairwise. 
(The associative binary operation remains noncommutative.)
This raises the natural question of enumerating the number of equivalence classes of monomials under the congruence
generated by these relations. To that end, we choose a canonical representative for each equivalence class. Thus, we define $V_d^c \subseteq V_d$ to be the smallest set such that
\begin{itemize} 
\item $\ast\in V_d^c$;
\item  if $v_1,v_2\in V_d^c$, then $v_1v_2 \in V_d^c$;
\item if $v\in V_d^c$ can be written as $v_1v_2$ for $v_1,v_2 \in V_d^c$, then $P_i(v)\in V_d^c$ for every $i \in [d]$;
\item if $v\in V_d^c$ has the form $P_j(v')$ for some $v' \in V_d^c$ and $j \in [d]$, 
then $P_i(v)\in V_d^c$ for every $i \in [j]$. In other words, subwords of the form $P_i(P_j(v'))$ are allowed if and only if $i \le j$.
\end{itemize} 

Then we have the following.

\begin{lemma}
Every equivalence class of monomials in $V_d$ under the relations
\[
P_i(P_j(v)) \sim P_j(P_i(v))
\qquad (1 \le i < j \le d,\; v \in V_d)
\]
contains a unique representative in $V_d^c$.
\end{lemma}

\begin{proof}
Given any monomial in $V_d$, each maximal chain of nested unary operators may be reordered uniquely into weakly increasing order of indices, using the commutativity relations. Applying this procedure independently to every maximal unary chain produces an element of $V_d^c$. Uniqueness follows because the multiset of unary labels on each maximal unary chain is preserved under the relations, and weakly increasing order is unique.
\end{proof}

Thus, $V_d^c$ is the set of monomials in which each maximal chain of unary operators is 
arranged in weakly increasing order in terms of their indices. Observe that every equivalence class in $V_d$ under relation \eqref{eq:sec3:eqrel} contains exactly one element from $V_d^c$. For example,  in Figure~\ref{fig:a221boxes}, monomials in the same column within the same box are equivalent when $P_1$ and $P_2$ commute. Moreover, the top element in each column belongs to $V_2^c$. 

In this section, we again proceed in three steps: first by degree and operator multiplicity, then by length gradings, and finally through several combinatorial interpretations.

\subsection{Degree and multiplicity grading}

For $r\ge 1$ and $\mathbf{s}=(s_1,\ldots,s_d)\in\mathbb{Z}_{\ge0}^d$, let
\[
a_d^c(r;\mathbf{s})
\]
denote the number of elements in $V_d^c$ with degree $r$ and multiplicity vector $\mathbf{s}$. For example, one can infer from Figure~\ref{fig:a221boxes} (by counting total columns within boxes) that $a_2^c(2; (2,1)) = 18$.  We encode these numbers in the multivariate generating function
\[
A_d^c(u; \mathbf{p})
:=
\sum_{r\ge1}\;\sum_{s_1,\ldots,s_d\ge0}
a_d^c(r;\mathbf{s})\,u^r p_1^{s_1}\cdots p_d^{s_d}.
\]
To derive the functional equation for $A_d^c$, we next introduce two auxiliary classes. First, recall that an atom is a monomial that cannot be written as a product of two monomials. Next, we say that a monomial is a \emph{non-unary root} if it is either the single indeterminate $\ast$ or a product of at least two atoms. This is precisely the class of monomials that may occur beneath a nonempty canonical chain of unary operators. Then we let
\[
\overline a_d^c(r;\mathbf{s}), \qquad \widetilde a_d^c(r;\mathbf{s})
\]
respectively denote the number of atoms and non-unary roots in $V_d^c$ with degree $r$ and multiplicities $s_1, \ldots, s_d$. 
Define 
\begin{align*}
\overline A_d^c(u;\mathbf{p})
&:=
\sum_{r\ge1}\;\sum_{s_1,\ldots,s_d\ge0}
\overline a_d^c(r;\mathbf{s})\,u^r p_1^{s_1}\cdots p_d^{s_d},
\\
\widetilde A_d^c(u;\mathbf{p})
&:=
\sum_{r\ge1}\;\sum_{s_1,\ldots,s_d\ge0}
\widetilde a_d^c(r;\mathbf{s})\,u^r p_1^{s_1}\cdots p_d^{s_d},
\end{align*}
to be the corresponding generating functions. Because multiplication is still noncommutative, every monomial in $V_d^c$ factors
uniquely as a sequence of atoms.
Thus
\begin{equation}\label{eq:sec3seq}
A_d^c(u;\mathbf{p})
=
\sum_{j \ge1}\bigl(\,\overline A_d^c(u;\mathbf{p})\bigr)^j
=
\frac{\overline A_d^c(u;\mathbf{p})}{1-\overline A_d^c(u;\mathbf{p})}.
\end{equation}
Likewise, a non-unary root monomial is either the single indeterminate $\ast$ or a
product of at least two atoms, so
\begin{equation}\label{eq:sec3nonunary}
\widetilde A_d^c(u;\mathbf{p})
=u+\sum_{j \ge2}\bigl(\,\overline A_d^c(u;\mathbf{p})\bigr)^j
=u+\frac{\overline A_d^c(u;\mathbf{p})^2}{1-\overline A_d^c(u;\mathbf{p})}
=u+A_d^c(u;\mathbf{p})-\overline A_d^c(u;\mathbf{p}).
\end{equation}

We next analyze the atoms in the commutative setting. 
Observe that an atom in $V_d^c$ is either the single indeterminate $\ast$, or has the form
\[
P_{1}^{s_1} \left( P_{2}^{s_2} \left( \cdots P_{d}^{s_d} (v) \cdots \right)\right)
\]
where $v$ is a non-unary root in $V_d^c$, and $\mathbf{s} = (s_1, \ldots, s_d)$ is a nonzero vector in $\mathbb{Z}_{\geq 0}^d$. Since
\[
\sum_{\mathbf{s} \in \mathbb{Z}_{\geq 0}^d \setminus \{\mathbf{0}\}} p_1^{s_1}\cdots p_d^{s_d}
=
\prod_{i=1}^d \frac{1}{1-p_i}-1,
\]
we obtain that
\begin{equation}\label{eq:sec3atom-pre}
\overline A_d^c(u;\mathbf{p})
=u+
\left(\prod_{i=1}^d\frac{1}{1-p_i}-1\right)
\widetilde A_d^c(u;\mathbf{p}).
\end{equation}
Substituting~\eqref{eq:sec3nonunary} into~\eqref{eq:sec3atom-pre} and simplifying yields
a much cleaner formula.

\begin{lemma}\label{lem:sec3atom}
Let $\Delta(\mathbf{p}):=1-\prod_{i=1}^d (1-p_i)$. Then
\begin{equation}\label{eq:sec3atom}
\overline A_d^c(u;\mathbf{p})
=u+\Delta(\mathbf{p})\,A_d^c(u;\mathbf{p}).
\end{equation}
\end{lemma}

\begin{proof}
From~\eqref{eq:sec3nonunary} we have $\widetilde A_d^c = u + A_d^c - \overline A_d^c$. Substituting into~\eqref{eq:sec3atom-pre} gives
\[
\overline A_d^c
=u+
\left(\prod_{i=1}^d\frac{1}{1-p_i}-1\right)
\bigl(u+A_d^c-\overline A_d^c\bigr).
\]
Multiplying both sides by $\prod_{i=1}^d (1-p_i)$ and simplifying yields
\[
\left( \prod_{i=1}^d (1-p_i) \right) \overline A_d^c
= u+
\left(1 -  \prod_{i=1}^d (1-p_i) \right)  A_d^c + \left( \prod_{i=1}^d (1-p_i)  -1 \right)  \overline A_d^c.
\]
Rearranging the above gives~\eqref{eq:sec3atom}.
\end{proof}

Combining~\eqref{eq:sec3seq} and~\eqref{eq:sec3atom}, we obtain the master functional
equation.

\begin{theorem}\label{thm:sec3master}
Let $\Delta(\mathbf{p}):=1-\prod_{i=1}^d (1-p_i)$.  Then the multivariate generating function $A_d^c(u;\mathbf{p})$ satisfies
\begin{equation}\label{eq:sec3masterfunc}
A_d^c(u;\mathbf{p})
=u+uA_d^c(u;\mathbf{p})
+\Delta(\mathbf{p}) \left( A_d^c(u;\mathbf{p})
+A_d^c(u;\mathbf{p})^2\right).
\end{equation}
\end{theorem}

\begin{proof}
From~\eqref{eq:sec3atom} we have $\overline A_d^c = u + \Delta(\mathbf{p}) A_d^c$. Substituting 
into~\eqref{eq:sec3seq} gives
\[
A_d^c
=
\frac{u + \Delta(\mathbf{p}) A_d^c}{1-u - \Delta(\mathbf{p}) A_d^c}. 
\]
Multiplying both sides by $1-u - \Delta(\mathbf{p}) A_d^c$ and rearranging gives~\eqref{eq:sec3masterfunc}.
\end{proof}

Unlike the noncommuting unary case, the coefficients of $A_d^c$ do not seem to admit a
simple closed form in general.
Nevertheless, the multigraded functional equation still yields an effective recurrence.

For each nonempty subset $J\subseteq[d]$, let $\mathbf e_J\in\{0,1\}^d$ denote its
indicator vector.
Since
\[
\Delta(\mathbf{p})
=
1-\prod_{i=1}^d(1-p_i)
=
\sum_{\emptyset\neq J\subseteq[d]}
(-1)^{|J|+1}\mathbf{p}^{\mathbf e_J}
\]
(where the exponentiation $\mathbf{p}^{\mathbf e_J}$ is performed coordinatewise), 
extracting coefficients from~\eqref{eq:sec3masterfunc} gives the following result.

\begin{theorem}\label{thm:sec3multirec}
For every $\mathbf{s}\in\mathbb{Z}_{\ge0}^d$, we have $a_d^c(1;\mathbf{s})=1$.
For $r\ge2$,
\begin{align}
&a_d^c(r;\mathbf{s})
=
a_d^c(r-1;\mathbf{s}) \notag
\\
&\quad+
\sum_{\emptyset\neq J\subseteq[d]}
(-1)^{|J|+1}
\left(
a_d^c(r;\mathbf{s}-\mathbf e_J)
+
\sum_{i=1}^{r-1}
\sum_{\mathbf{0} \leq \boldsymbol{\alpha}\le \mathbf{s}-\mathbf e_J}
a_d^c(i;\boldsymbol{\alpha})\,
a_d^c(r-i;\mathbf{s}-\mathbf e_J-\boldsymbol{\alpha})
\right), \label{eq:sec3multirec}
\end{align}
where all vector inequalities are understood coordinatewise, and we use the convention
that $a_d^c(r;\mathbf{s})=0$ whenever $r\le0$ or some coordinate of $\mathbf{s}$ is
negative.
\end{theorem}

\begin{proof}
The initial condition $a_d^c(1;\mathbf{s})=1$ is immediate: for fixed multiplicity vector
$\mathbf{s} \geq \mathbf{0}$ there is exactly one canonical commuting monomial of degree $1$, namely
\[
P_1^{s_1}\left( \cdots \left( P_d^{s_d}(\ast) \right) \cdots \right),
\]
where the unary operators are arranged in the chosen canonical order.
Next, suppose $r\ge2$. From~\eqref{eq:sec3masterfunc}, we have
\begin{align*}
a_d^c(r; \mathbf{s}) &= [u^r p_1^{s_1}\cdots p_d^{s_d}] A_d^c(u;\mathbf{p})\\
&= [u^r p_1^{s_1}\cdots p_d^{s_d}] \left( u+uA_d^c(u;\mathbf{p})
+\Delta(\mathbf{p}) \left( A_d^c(u;\mathbf{p})
+A_d^c(u;\mathbf{p})^2\right) \right).
\end{align*}
Now $[u^r p_1^{s_1}\cdots p_d^{s_d}] u = 0$ (since $r \geq 2$) and $[u^r p_1^{s_1}\cdots p_d^{s_d}] uA_d^c(u;\mathbf{p}) = a_d^c(r-1;\mathbf{s})$. Also,
\begin{align*}
&[u^r p_1^{s_1}\cdots p_d^{s_d}] \Delta(\mathbf{p}) \left( A_d^c(u;\mathbf{p})+A_d^c(u;\mathbf{p})^2\right)\\
={}& [u^r p_1^{s_1}\cdots p_d^{s_d}] \sum_{\emptyset\neq J\subseteq[d]}
(-1)^{|J|+1}\mathbf{p}^{\mathbf e_J} \left( A_d^c(u;\mathbf{p})+A_d^c(u;\mathbf{p})^2\right),
\end{align*}
which yields the sum on the right-hand side of~\eqref{eq:sec3multirec}.
\end{proof}

\subsection{Length-graded specializations}

We now turn to the length-graded sequences associated with $V_d^c$.  As in Section~\ref{sec2}, we fix $\ell \geq 1$ and map monomials in $V_d^c$ to bracketed words while assigning length $\ell$ to $\ast$. We define
\[
b_{d,\ell}^c(n)
:=
\sum_{\substack{r\ge1,\;\mathbf{s}\ge\mathbf{0}\\ \ell r+2|\mathbf{s}|=n}}
a_d^c(r;\mathbf{s}),
\qquad
B_{d,\ell}^c(z):=\sum_{n\ge1} b_{d,\ell}^c(n)z^n.
\]
Observe that
\[
B_{d,\ell}^c(z)=A_d^c(z^{\ell} ;z^2,\ldots,z^2).
\]
Observe that if we set $u=z^{\ell}$ and $p_1=\cdots=p_d=z^2$, then 
\[
\Delta(\mathbf{p}) = 1 - \prod_{i=1}^d (1-p_i) = 1 - (1-z^2)^d,
\]
and therefore~\eqref{eq:sec3masterfunc} specializes to
\begin{equation}\label{eq:sec3b11func}
B_{d,\ell}^c(z)
=
z^{\ell}+z^{\ell}B_{d,\ell}^c(z)
+\bigl(1-(1-z^2)^d\bigr)
\left(B_{d,\ell}^c(z)+B_{d,\ell}^c(z)^2\right).
\end{equation}
This gives the following.

\begin{theorem}\label{thm:sec3bdlcgeneral}
The sequence $b_{d,\ell}^c(n)$ satisfies
\[
b_{d,\ell}^c(\ell)=1,
\]
and for every $n\ge \ell +1$,
\begin{align}
\label{eq:sec3b11rec}
&
b_{d,\ell}^c(n)
=
b_{d,\ell}^c(n-\ell)
\\
&
\notag
+
\sum_{j=1}^{\min(d,\lfloor n/2\rfloor)}
(-1)^{j+1}\binom{d}{j}
\left(
b_{d,\ell}^c(n-2j)
+
\sum_{r=1}^{n-2j-1} b_{d,\ell}^c(r)\,b_{d,\ell}^c(n-2j-r)
\right),
\end{align}
with the convention $b_{d,\ell}^c(j)=0$ for $j \le \ell-1$.
\end{theorem}

\begin{proof}
Since
\[
1-(1-z^2)^d
=
\sum_{j=1}^d (-1)^{j+1}\binom{d}{j}z^{2j},
\]
comparing coefficients of $z^n$ in~\eqref{eq:sec3b11func} gives \eqref{eq:sec3b11rec}.
\end{proof}

\begin{table}[ht]
\centering
\begin{tabular}{c|l|l}
$d$ & First terms of $b_{d,1}^c(n)$ & OEIS entry \\
\hline
\\[-12pt]
1 & $1,1,2,4,8,17,37,82,185,423,978,2283, 5373, 12735, 30372,\ldots$ & \seqnum{A004148} \\
2 & $1,1,3,7,16,42,109,289,787,2162,6013,16901, 47872,\ldots$ & \seqnum{A394942}\textsuperscript{*} \\
3 & $1,1,4,10,25,76,218,649,1996,6140,19207,60715, 193194,\ldots$ & \seqnum{A394943}\textsuperscript{*} \\
4 & $1, 1, 5, 13, 35, 119, 365, 1189, 4002, 13360, 45659, 157383, \ldots$ & \seqnum{A394944}\textsuperscript{*}
\end{tabular}
\caption{First terms of the sequences $b_{d,1}^c(n)$ for $d=1,2,3,4$.}
\label{tab:bcd1}
\end{table}

Table~\ref{tab:bcd1} records the first terms of $b_{d,1}^c(n)$ for $d \leq 4$.  
For $d=1$ this agrees with the generalized Catalan numbers \seqnum{A004148}, while the cases $d=2,3,4$ are among the new OEIS submissions arising from this project. 

On the other hand, it is easy to see that
$b_{1,\ell}^c(n) = b_{1,\ell}(n)$
for all $n, \ell \geq 1$. This is because when there is only one unary operator, 
commutativity is trivial, and we have $V_1 = V_1^c$.

\begin{table}[ht]
\centering
\begin{tabular}{c|l|l}
$d$ & First terms of $b_{d,2}^c(2n)$ & OEIS entry \\
\hline
\\[-12pt]
1 & $1,2,5,14,42,132,429,1430,4862,16796, 58786, 208012,\ldots$ & \seqnum{A000108} \\
2 & $1,3,10,38,156,673,3007,13792,64559,307153, 1480946,\ldots$ &\seqnum{A394945}\textsuperscript{*} \\
3 & $1,4,16,74,373,1978,10869,61333,353295,2068790,\ldots$ & \seqnum{A394946}\textsuperscript{*} \\
4 & $1, 5, 23, 123, 724, 4486, 28801, 189933, 1278791, 8753322, \ldots$ & \seqnum{A394947}\textsuperscript{*}
\end{tabular}
\caption{First terms of the sequences $b_{d,2}^c(2n)$ for $d=1,2,3,4$.}
\label{tab:bcd2}
\end{table}

Table~\ref{tab:bcd2} records the first terms of $b_{d,2}^c(2n)$ for $d \leq 4$.  
In particular, $b_{1,2}^c(2n)$ coincides with $b_{1,2}(2n)$ and again specializes 
to the Catalan numbers \seqnum{A000108}. 

\subsection{Combinatorial interpretations}\label{sec33}

The combinatorial models from Section~\ref{sec2} admit natural commutative analogs.
Since the underlying recursive constructions are the same, we do not repeat the maps in
full.
Instead, we record the extra condition introduced by the commutativity of the unary
operators: every maximal unary chain is canonically ordered, and this appears
combinatorially as a weak monotonicity condition on labels.

\subsubsection{Rooted ordered trees with monotone unary chains}

As in Section~\ref{sec2}, every monomial may be represented by a rooted ordered tree in which leaves represent occurrences of the indeterminate $\ast$,
unary internal nodes are labeled by elements of $[d]$, and unlabeled internal nodes of
arity at least $2$ represent nontrivial products.
Because multiplication remains noncommutative, the children of every non-unary internal
node remain ordered from left to right.

Now by restricting the domain of this map to the subset $V_d^c$, the labels along each maximal chain of unary vertices are weakly increasing when read from the root toward the leaves. This gives the following.

\begin{proposition}\label{prop:sec33trees}
For integers $r\ge1$ and $\mathbf{s}=(s_1,\ldots,s_d)\in\mathbb{Z}_{\ge0}^d$, the
number $a_d^c(r;\mathbf{s})$ counts rooted ordered trees with exactly $r$ leaves and
exactly $s_i$ unary internal nodes labeled $i$, for each $i\in[d]$, in which every
non-unary internal node has at least two children and every maximal unary chain has
weakly increasing labels.
\end{proposition}

\begin{proof}
Use the same tree construction as in Proposition~\ref{prop:sec23trees}.  The only
difference is that in the commuting case we work with canonical representatives, so each
maximal chain of unary operators is written in nondecreasing order.  This is exactly the
stated monotonicity condition on the corresponding unary chains in the tree.
\end{proof}

\subsubsection{Peakless lattice paths with monotone ascents}

Likewise, the path model from Section~\ref{sec2} carries over unchanged except for the same monotonicity phenomenon. Opening brackets of type $i$ still correspond to up-steps $U_i=(1,1)$, closing brackets to down-steps $D=(1,-1)$, and the realization of $\ast$ to a horizontal step $H_\ell=(\ell,0)$.
As before, the resulting paths remain on or above the $x$-axis, end on the $x$-axis, and avoid peaks $UD$.

For an up-step $U_i$ from height $j$ to height $j+1$, we define its \emph{matching down step} to be the first subsequent down-step from height $j+1$ to height $j$.
Thus a matching pair of steps corresponds to the opening and closing brackets of a single unary operator. We call a maximal consecutive block of up-steps a \emph{matched ascent}  if the matching down-steps of all those up-steps belong to a single descent. For example, the initial ascent $UUU$ is a matched ascent in $UUUHHDDD$, but not in $UUUHDHDD$, since its three matching down-steps do not belong to the same descent.

In the commuting case, the labels along every matched ascent are weakly increasing.

\begin{proposition}\label{prop:sec33pathsgeneral}
Let $\ell\ge1$ and $n\ge1$.
Then $b_{d,\ell}^c(n)$ counts lattice paths of total horizontal length $n$ that
\begin{itemize}
\item start at $(0,0)$ and end on the $x$-axis;
\item use only steps $U_i=(1,1)$ for $1\le i\le d$, $D=(1,-1)$, and
$H_\ell=(\ell,0)$;
\item remain on or above the $x$-axis;
\item avoid peaks, that is, do not contain an instance of $UD$;
\item have weakly increasing labels along every matched ascent.
\end{itemize}
\end{proposition}

\begin{proof}
Apply the same bracketed-word-to-path map as in
Proposition~\ref{prop:sec23pathsgeneral}.  The first four properties are unchanged.
The additional condition comes from the fact that in $V_d^c$ every maximal chain of unary operators is written in canonical nondecreasing order, which translates exactly into weakly
increasing labels along each matched ascent.
\end{proof}

For example, Figure~\ref{fig:paths} illustrates the 21 lattice paths counted by $b_{2,3}^c(10)$.

\begin{figure}[htbp]
\centering
\def\sc{0.4}
\def\grid{\foreach \x in {\xlb ,...,\xub}
{    \ifthenelse{\NOT 0 = \x}{\draw[thick](\x ,-1pt) -- (\x ,1pt);}{}
\draw[dotted](\x,\ylb- \buf) -- (\x,\yub + \buf);}
\foreach \y in {\ylb ,...,\yub}
{    \ifthenelse{\NOT 0 = \y}{\draw[thick](-1pt, \y) -- (1pt, \y);}{}
\draw[dotted](\xlb- \buf, \y) -- (\xub + \buf, \y);}
}

\begin{tabular}{ccc}
\begin{tikzpicture}[scale = \sc, font=\scriptsize\sffamily, thick,main node/.style={circle,inner sep=0.4mm,draw, fill}]
\def\xlb{0}; \def\xub{10}; \def\ylb{0}; \def\yub{3}; \def\buf{0}; \grid
\node[main node] at (0,0) (0) {};
\node[main node] at (1,1) (1) {};
\node[main node] at (2,2) (2) {};
\node[main node] at (5,2) (3) {};
\node[main node] at (8,2) (4) {};
\node[main node] at (9,1) (5) {};
\node[main node] at (10,0) (6) {};
\draw[ultra thick] 
(0) --node[pos=0.5, below right, inner sep=0pt] {$1$} 
(1) --node[pos=0.5, below right, inner sep=0pt] {$1$}
(2) -- (3)-- (4)-- (5)-- (6);
\end{tikzpicture}
&
\begin{tikzpicture}[scale = \sc, font=\scriptsize\sffamily, thick,main node/.style={circle,inner sep=0.4mm,draw, fill}]
\def\xlb{0}; \def\xub{10}; \def\ylb{0}; \def\yub{3}; \def\buf{0}; \grid
\node[main node] at (0,0) (0) {};
\node[main node] at (1,1) (1) {};
\node[main node] at (2,2) (2) {};
\node[main node] at (5,2) (3) {};
\node[main node] at (8,2) (4) {};
\node[main node] at (9,1) (5) {};
\node[main node] at (10,0) (6) {};
\draw[ultra thick] 
(0) --node[pos=0.5, below right, inner sep=0pt] {$1$} 
(1) --node[pos=0.5, below right, inner sep=0pt] {$2$}
(2) -- (3)-- (4)-- (5)-- (6);
\end{tikzpicture}
&
\begin{tikzpicture}[scale = \sc, font=\scriptsize\sffamily, thick,main node/.style={circle,inner sep=0.4mm,draw, fill}]
\def\xlb{0}; \def\xub{10}; \def\ylb{0}; \def\yub{3}; \def\buf{0}; \grid
\node[main node] at (0,0) (0) {};
\node[main node] at (1,1) (1) {};
\node[main node] at (2,2) (2) {};
\node[main node] at (5,2) (3) {};
\node[main node] at (8,2) (4) {};
\node[main node] at (9,1) (5) {};
\node[main node] at (10,0) (6) {};
\draw[ultra thick] 
(0) --node[pos=0.5, below right, inner sep=0pt] {$2$} 
(1) --node[pos=0.5, below right, inner sep=0pt] {$2$}
(2) -- (3)-- (4)-- (5)-- (6);
\end{tikzpicture}
\\[0.8em]

\begin{tikzpicture}[scale = \sc, font=\scriptsize\sffamily, thick,main node/.style={circle,inner sep=0.4mm,draw, fill}]
\def\xlb{0}; \def\xub{10}; \def\ylb{0}; \def\yub{3}; \def\buf{0}; \grid
\node[main node] at (0,0) (0) {};
\node[main node] at (3,0) (1) {};
\node[main node] at (4,1) (2) {};
\node[main node] at (5,2) (3) {};
\node[main node] at (8,2) (4) {};
\node[main node] at (9,1) (5) {};
\node[main node] at (10,0) (6) {};
\draw[ultra thick] 
(0) -- (1)
--node[pos=0.5, below right, inner sep=0pt] {$1$} (2)
--node[pos=0.5, below right, inner sep=0pt] {$1$} (3)
-- (4)-- (5)-- (6);
\end{tikzpicture}
&
\begin{tikzpicture}[scale = \sc, font=\scriptsize\sffamily, thick,main node/.style={circle,inner sep=0.4mm,draw, fill}]
\def\xlb{0}; \def\xub{10}; \def\ylb{0}; \def\yub{3}; \def\buf{0}; \grid
\node[main node] at (0,0) (0) {};
\node[main node] at (3,0) (1) {};
\node[main node] at (4,1) (2) {};
\node[main node] at (5,2) (3) {};
\node[main node] at (8,2) (4) {};
\node[main node] at (9,1) (5) {};
\node[main node] at (10,0) (6) {};
\draw[ultra thick] 
(0) -- (1)
--node[pos=0.5, below right, inner sep=0pt] {$1$} (2)
--node[pos=0.5, below right, inner sep=0pt] {$2$} (3)
-- (4)-- (5)-- (6);
\end{tikzpicture}
&
\begin{tikzpicture}[scale = \sc, font=\scriptsize\sffamily, thick,main node/.style={circle,inner sep=0.4mm,draw, fill}]
\def\xlb{0}; \def\xub{10}; \def\ylb{0}; \def\yub{3}; \def\buf{0}; \grid
\node[main node] at (0,0) (0) {};
\node[main node] at (3,0) (1) {};
\node[main node] at (4,1) (2) {};
\node[main node] at (5,2) (3) {};
\node[main node] at (8,2) (4) {};
\node[main node] at (9,1) (5) {};
\node[main node] at (10,0) (6) {};
\draw[ultra thick] 
(0) -- (1)
--node[pos=0.5, below right, inner sep=0pt] {$2$} (2)
--node[pos=0.5, below right, inner sep=0pt] {$2$} (3)
-- (4)-- (5)-- (6);
\end{tikzpicture}
\\[0.8em]

\begin{tikzpicture}[scale = \sc, font=\scriptsize\sffamily, thick,main node/.style={circle,inner sep=0.4mm,draw, fill}]
\def\xlb{0}; \def\xub{10}; \def\ylb{0}; \def\yub{3}; \def\buf{0}; \grid
\node[main node] at (0,0) (0) {};
\node[main node] at (1,1) (1) {};
\node[main node] at (2,2) (2) {};
\node[main node] at (5,2) (3) {};
\node[main node] at (6,1) (4) {};
\node[main node] at (7,0) (5) {};
\node[main node] at (10,0) (6) {};
\draw[ultra thick] 
(0) --node[pos=0.5, below right, inner sep=0pt] {$1$} (1)
--node[pos=0.5, below right, inner sep=0pt] {$1$} (2)
-- (3)-- (4)-- (5)-- (6);
\end{tikzpicture}
&
\begin{tikzpicture}[scale = \sc, font=\scriptsize\sffamily, thick,main node/.style={circle,inner sep=0.4mm,draw, fill}]
\def\xlb{0}; \def\xub{10}; \def\ylb{0}; \def\yub{3}; \def\buf{0}; \grid
\node[main node] at (0,0) (0) {};
\node[main node] at (1,1) (1) {};
\node[main node] at (2,2) (2) {};
\node[main node] at (5,2) (3) {};
\node[main node] at (6,1) (4) {};
\node[main node] at (7,0) (5) {};
\node[main node] at (10,0) (6) {};
\draw[ultra thick] 
(0) --node[pos=0.5, below right, inner sep=0pt] {$1$} (1)
--node[pos=0.5, below right, inner sep=0pt] {$2$} (2)
-- (3)-- (4)-- (5)-- (6);
\end{tikzpicture}
&
\begin{tikzpicture}[scale = \sc, font=\scriptsize\sffamily, thick,main node/.style={circle,inner sep=0.4mm,draw, fill}]
\def\xlb{0}; \def\xub{10}; \def\ylb{0}; \def\yub{3}; \def\buf{0}; \grid
\node[main node] at (0,0) (0) {};
\node[main node] at (1,1) (1) {};
\node[main node] at (2,2) (2) {};
\node[main node] at (5,2) (3) {};
\node[main node] at (6,1) (4) {};
\node[main node] at (7,0) (5) {};
\node[main node] at (10,0) (6) {};
\draw[ultra thick] 
(0) --node[pos=0.5, below right, inner sep=0pt] {$2$} (1)
--node[pos=0.5, below right, inner sep=0pt] {$2$} (2)
-- (3)-- (4)-- (5)-- (6);
\end{tikzpicture}
\\[0.8em]

\begin{tikzpicture}[scale = \sc, font=\scriptsize\sffamily, thick,main node/.style={circle,inner sep=0.4mm,draw, fill}]
\def\xlb{0}; \def\xub{10}; \def\ylb{0}; \def\yub{3}; \def\buf{0}; \grid
\node[main node] at (0,0) (0) {};
\node[main node] at (1,1) (1) {};
\node[main node] at (4,1) (2) {};
\node[main node] at (5,2) (3) {};
\node[main node] at (8,2) (4) {};
\node[main node] at (9,1) (5) {};
\node[main node] at (10,0) (6) {};
\draw[ultra thick] 
(0) --node[pos=0.5, below right, inner sep=0pt] {$1$} (1)
-- (2)
--node[pos=0.5, below right, inner sep=0pt] {$1$} (3)
-- (4)-- (5)-- (6);
\end{tikzpicture}
&
\begin{tikzpicture}[scale = \sc, font=\scriptsize\sffamily, thick,main node/.style={circle,inner sep=0.4mm,draw, fill}]
\def\xlb{0}; \def\xub{10}; \def\ylb{0}; \def\yub{3}; \def\buf{0}; \grid
\node[main node] at (0,0) (0) {};
\node[main node] at (1,1) (1) {};
\node[main node] at (4,1) (2) {};
\node[main node] at (5,2) (3) {};
\node[main node] at (8,2) (4) {};
\node[main node] at (9,1) (5) {};
\node[main node] at (10,0) (6) {};
\draw[ultra thick] 
(0) --node[pos=0.5, below right, inner sep=0pt] {$1$} (1)
-- (2)
--node[pos=0.5, below right, inner sep=0pt] {$2$} (3)
-- (4)-- (5)-- (6);
\end{tikzpicture}
&
\begin{tikzpicture}[scale = \sc, font=\scriptsize\sffamily, thick,main node/.style={circle,inner sep=0.4mm,draw, fill}]
\def\xlb{0}; \def\xub{10}; \def\ylb{0}; \def\yub{3}; \def\buf{0}; \grid
\node[main node] at (0,0) (0) {};
\node[main node] at (1,1) (1) {};
\node[main node] at (4,1) (2) {};
\node[main node] at (5,2) (3) {};
\node[main node] at (8,2) (4) {};
\node[main node] at (9,1) (5) {};
\node[main node] at (10,0) (6) {};
\draw[ultra thick] 
(0) --node[pos=0.5, below right, inner sep=0pt] {$2$} (1)
-- (2)
--node[pos=0.5, below right, inner sep=0pt] {$1$} (3)
-- (4)-- (5)-- (6);
\end{tikzpicture}
\\[0.8em]

\begin{tikzpicture}[scale = \sc, font=\scriptsize\sffamily, thick,main node/.style={circle,inner sep=0.4mm,draw, fill}]
\def\xlb{0}; \def\xub{10}; \def\ylb{0}; \def\yub{3}; \def\buf{0}; \grid
\node[main node] at (0,0) (0) {};
\node[main node] at (1,1) (1) {};
\node[main node] at (4,1) (2) {};
\node[main node] at (5,2) (3) {};
\node[main node] at (8,2) (4) {};
\node[main node] at (9,1) (5) {};
\node[main node] at (10,0) (6) {};
\draw[ultra thick] 
(0) --node[pos=0.5, below right, inner sep=0pt] {$2$} (1)
-- (2)
--node[pos=0.5, below right, inner sep=0pt] {$2$} (3)
-- (4)-- (5)-- (6);
\end{tikzpicture}
&
\begin{tikzpicture}[scale = \sc, font=\scriptsize\sffamily, thick,main node/.style={circle,inner sep=0.4mm,draw, fill}]
\def\xlb{0}; \def\xub{10}; \def\ylb{0}; \def\yub{3}; \def\buf{0}; \grid
\node[main node] at (0,0) (0) {};
\node[main node] at (1,1) (1) {};
\node[main node] at (4,1) (2) {};
\node[main node] at (5,0) (3) {};
\node[main node] at (6,1) (4) {};
\node[main node] at (9,1) (5) {};
\node[main node] at (10,0) (6) {};
\draw[ultra thick] 
(0) --node[pos=0.5, below right, inner sep=0pt] {$1$} (1)
-- (2)-- (3)
--node[pos=0.5, below right, inner sep=0pt] {$1$} (4)
-- (5)-- (6);
\end{tikzpicture}
&
\begin{tikzpicture}[scale = \sc, font=\scriptsize\sffamily, thick,main node/.style={circle,inner sep=0.4mm,draw, fill}]
\def\xlb{0}; \def\xub{10}; \def\ylb{0}; \def\yub{3}; \def\buf{0}; \grid
\node[main node] at (0,0) (0) {};
\node[main node] at (1,1) (1) {};
\node[main node] at (4,1) (2) {};
\node[main node] at (5,0) (3) {};
\node[main node] at (6,1) (4) {};
\node[main node] at (9,1) (5) {};
\node[main node] at (10,0) (6) {};
\draw[ultra thick] 
(0) --node[pos=0.5, below right, inner sep=0pt] {$1$} (1)
-- (2)-- (3)
--node[pos=0.5, below right, inner sep=0pt] {$2$} (4)
-- (5)-- (6);
\end{tikzpicture}
\\[0.8em]

\begin{tikzpicture}[scale = \sc, font=\scriptsize\sffamily, thick,main node/.style={circle,inner sep=0.4mm,draw, fill}]
\def\xlb{0}; \def\xub{10}; \def\ylb{0}; \def\yub{3}; \def\buf{0}; \grid
\node[main node] at (0,0) (0) {};
\node[main node] at (1,1) (1) {};
\node[main node] at (4,1) (2) {};
\node[main node] at (5,0) (3) {};
\node[main node] at (6,1) (4) {};
\node[main node] at (9,1) (5) {};
\node[main node] at (10,0) (6) {};
\draw[ultra thick] 
(0) --node[pos=0.5, below right, inner sep=0pt] {$2$} (1)
-- (2)-- (3)
--node[pos=0.5, below right, inner sep=0pt] {$1$} (4)
-- (5)-- (6);
\end{tikzpicture}
&
\begin{tikzpicture}[scale = \sc, font=\scriptsize\sffamily, thick,main node/.style={circle,inner sep=0.4mm,draw, fill}]
\def\xlb{0}; \def\xub{10}; \def\ylb{0}; \def\yub{3}; \def\buf{0}; \grid
\node[main node] at (0,0) (0) {};
\node[main node] at (1,1) (1) {};
\node[main node] at (4,1) (2) {};
\node[main node] at (5,0) (3) {};
\node[main node] at (6,1) (4) {};
\node[main node] at (9,1) (5) {};
\node[main node] at (10,0) (6) {};
\draw[ultra thick] 
(0) --node[pos=0.5, below right, inner sep=0pt] {$2$} (1)
-- (2)-- (3)
--node[pos=0.5, below right, inner sep=0pt] {$2$} (4)
-- (5)-- (6);
\end{tikzpicture}
&
\begin{tikzpicture}[scale = \sc, font=\scriptsize\sffamily, thick,main node/.style={circle,inner sep=0.4mm,draw, fill}]
\def\xlb{0}; \def\xub{10}; \def\ylb{0}; \def\yub{3}; \def\buf{0}; \grid
\node[main node] at (0,0) (0) {};
\node[main node] at (1,1) (1) {};
\node[main node] at (2,2) (2) {};
\node[main node] at (5,2) (3) {};
\node[main node] at (6,1) (4) {};
\node[main node] at (9,1) (5) {};
\node[main node] at (10,0) (6) {};
\draw[ultra thick] 
(0) --node[pos=0.5, below right, inner sep=0pt] {$1$} (1)
--node[pos=0.5, below right, inner sep=0pt] {$1$} (2)
-- (3)-- (4)-- (5)-- (6);
\end{tikzpicture}
\\[0.8em]

\begin{tikzpicture}[scale = \sc, font=\scriptsize\sffamily, thick,main node/.style={circle,inner sep=0.4mm,draw, fill}]
\def\xlb{0}; \def\xub{10}; \def\ylb{0}; \def\yub{3}; \def\buf{0}; \grid
\node[main node] at (0,0) (0) {};
\node[main node] at (1,1) (1) {};
\node[main node] at (2,2) (2) {};
\node[main node] at (5,2) (3) {};
\node[main node] at (6,1) (4) {};
\node[main node] at (9,1) (5) {};
\node[main node] at (10,0) (6) {};
\draw[ultra thick] 
(0) --node[pos=0.5, below right, inner sep=0pt] {$1$} (1)
--node[pos=0.5, below right, inner sep=0pt] {$2$} (2)
-- (3)-- (4)-- (5)-- (6);
\end{tikzpicture}
&
\begin{tikzpicture}[scale = \sc, font=\scriptsize\sffamily, thick,main node/.style={circle,inner sep=0.4mm,draw, fill}]
\def\xlb{0}; \def\xub{10}; \def\ylb{0}; \def\yub{3}; \def\buf{0}; \grid
\node[main node] at (0,0) (0) {};
\node[main node] at (1,1) (1) {};
\node[main node] at (2,2) (2) {};
\node[main node] at (5,2) (3) {};
\node[main node] at (6,1) (4) {};
\node[main node] at (9,1) (5) {};
\node[main node] at (10,0) (6) {};
\draw[ultra thick] 
(0) --node[pos=0.5, below right, inner sep=0pt] {$2$} (1)
--node[pos=0.5, below right, inner sep=0pt] {$1$} (2)
-- (3)-- (4)-- (5)-- (6);
\end{tikzpicture}
&
\begin{tikzpicture}[scale = \sc, font=\scriptsize\sffamily, thick,main node/.style={circle,inner sep=0.4mm,draw, fill}]
\def\xlb{0}; \def\xub{10}; \def\ylb{0}; \def\yub{3}; \def\buf{0}; \grid
\node[main node] at (0,0) (0) {};
\node[main node] at (1,1) (1) {};
\node[main node] at (2,2) (2) {};
\node[main node] at (5,2) (3) {};
\node[main node] at (6,1) (4) {};
\node[main node] at (9,1) (5) {};
\node[main node] at (10,0) (6) {};
\draw[ultra thick] 
(0) --node[pos=0.5, below right, inner sep=0pt] {$2$} (1)
--node[pos=0.5, below right, inner sep=0pt] {$2$} (2)
-- (3)-- (4)-- (5)-- (6);
\end{tikzpicture}
\end{tabular}

\caption{The 21 peakless lattice paths counted by $b_{2,3}^c(10)$.
These use up-steps labeled by $1$ or $2$, down-steps, and horizontal steps of length $3$, with weakly increasing labels along every matched ascent.}\label{fig:paths}
\end{figure}
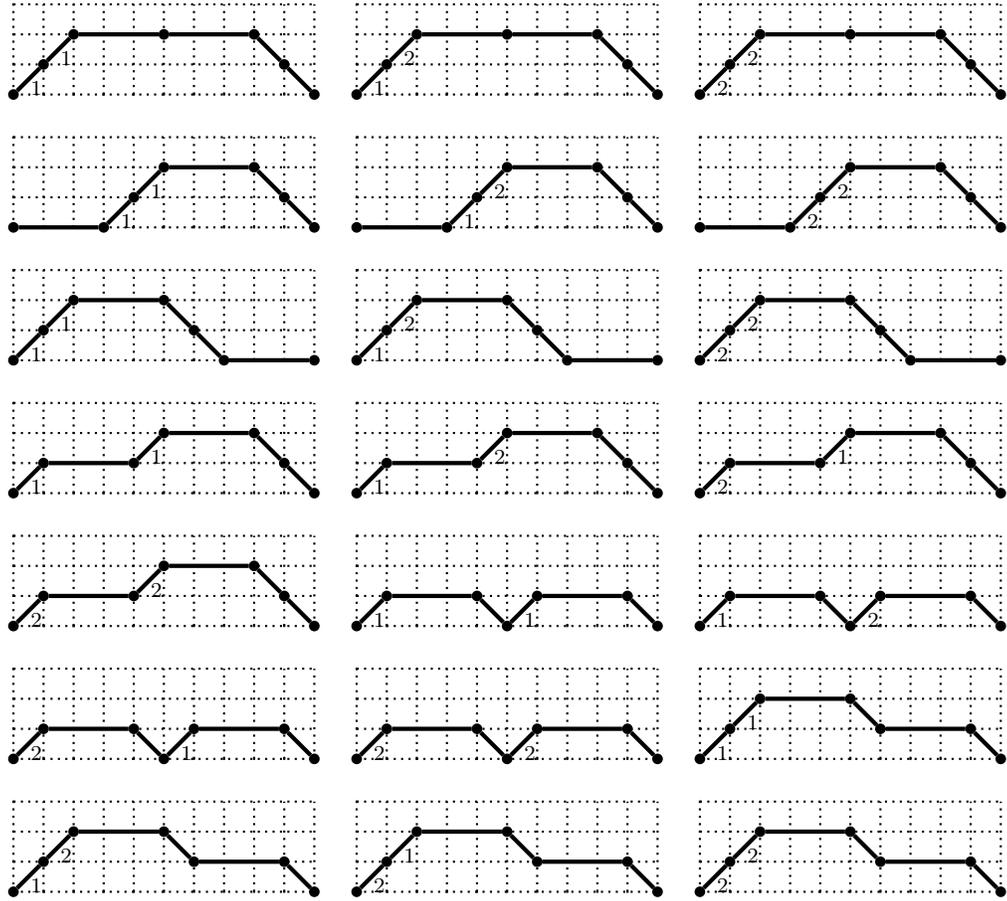

The cases $\ell = 1$ and $\ell = 2$ again give the two most relevant specializations.

\begin{corollary}\label{cor:sec33pathsspecial}
For every $n\ge1$,
\begin{enumerate}
\item[\textup{(i)}]
$b_{d,1}^c(n)$ counts peakless Motzkin paths of length $n$ with $d$ types of up-steps,
such that the labels along every matched ascent are weakly increasing.
\item[\textup{(ii)}]
$b_{d,2}^c(2n)$ counts peakless Schr\"oder paths of semilength $n$ with $d$ types of
up-steps, such that the labels along every matched ascent are weakly increasing.
\end{enumerate}
\end{corollary}

\subsubsection{Binary trees with labeled right edges}

Consider again rooted ordered binary trees where  left edges are unlabeled and right edges are labeled by elements of $[d]$. The commutativity condition is reflected by monotonicity along the right-edge chains that correspond to maximal chains of unary operators.

\begin{proposition}\label{prop:sec33binary}
For every $n\ge1$, the number $b_{d,2}^c(2n)$ counts rooted ordered binary trees with $n$ vertices in which every right edge carries one of $d$ possible labels and, whenever
$v_1,\ldots,v_k$ is a path such that each $\{v_i,v_{i+1}\}$ is a right edge and each intermediate vertex $v_2,\ldots,v_{k-1}$ has no left child, the labels along that chain are weakly increasing.
\end{proposition}

\begin{proof}
Use the same recursive map as in Proposition~\ref{prop:sec23binary}.  In the commuting case, every maximal chain of unary operators is first placed in canonical nondecreasing order, and under the binary-tree correspondence this is precisely the stated weak monotonicity condition along the corresponding right-edge chains.
\end{proof}

For comparison with the noncommutative case, note that the tree in Figure~\ref{fig:treetomonomial} is not represented in the present setting: the right-edge chain from the root to its grandchild has labels $2,1$, which are not weakly increasing, so this tree does not satisfy the condition in Proposition~\ref{prop:sec33binary}.

For example, Figure~\ref{fig:binarytrees} illustrates the 16 rooted ordered binary trees counted by $b_{3,2}^c(6)$. 

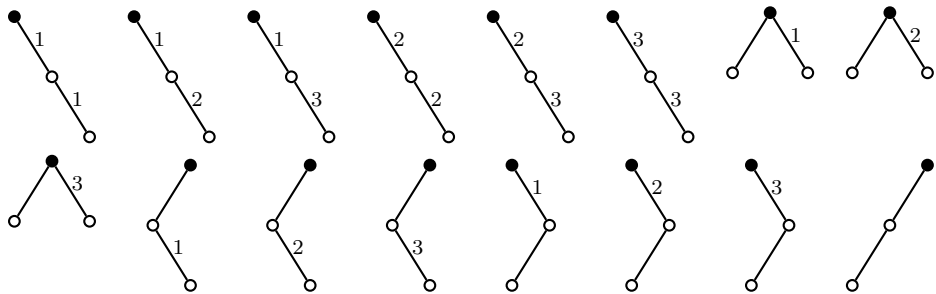
\begin{figure}[htbp]
\centering
\def\sc{0.5}
\def\ysc{0.8}
\begin{tabular}{cccccccc}

\begin{tikzpicture}[xscale=\sc, yscale=\ysc,thick,main node/.style={circle, inner sep=0.5mm, draw, font=\small\sffamily},edge label/.style={font=\scriptsize,  inner sep=0pt}]
\node[main node,fill] at (2,2) (0) {};
\node[main node] at (3,1) (1) {};
\node[main node] at (4,0) (2) {};
\path
(0) edge node[pos=0.5, edge label, above right] {$1$} (1)
(1) edge node[pos=0.5, edge label, above right] {$1$} (2);
\end{tikzpicture}
&

\begin{tikzpicture}[xscale=\sc, yscale=\ysc,thick,main node/.style={circle, inner sep=0.5mm, draw, font=\small\sffamily},edge label/.style={font=\scriptsize,  inner sep=0pt}]
\node[main node,fill] at (2,2) (0) {};
\node[main node] at (3,1) (1) {};
\node[main node] at (4,0) (2) {};
\path
(0) edge node[pos=0.5, edge label, above right] {$1$} (1)
(1) edge node[pos=0.5, edge label, above right] {$2$} (2);
\end{tikzpicture}
&

\begin{tikzpicture}[xscale=\sc, yscale=\ysc,thick,main node/.style={circle, inner sep=0.5mm, draw, font=\small\sffamily},edge label/.style={font=\scriptsize,  inner sep=0pt}]
\node[main node, fill] at (2,2) (0) {};
\node[main node] at (3,1) (1) {};
\node[main node] at (4,0) (2) {};
\path
(0) edge node[pos=0.5, edge label, above right] {$1$} (1)
(1) edge node[pos=0.5, edge label, above right] {$3$} (2);
\end{tikzpicture}
&

\begin{tikzpicture}[xscale=\sc, yscale=\ysc,thick,main node/.style={circle, inner sep=0.5mm, draw, font=\small\sffamily},edge label/.style={font=\scriptsize,  inner sep=0pt}]
\node[main node,fill] at (2,2) (0) {};
\node[main node] at (3,1) (1) {};
\node[main node] at (4,0) (2) {};
\path
(0) edge node[pos=0.5, edge label, above right] {$2$} (1)
(1) edge node[pos=0.5, edge label, above right] {$2$} (2);
\end{tikzpicture}
&

\begin{tikzpicture}[xscale=\sc, yscale=\ysc,thick,main node/.style={circle, inner sep=0.5mm, draw, font=\small\sffamily},edge label/.style={font=\scriptsize,  inner sep=0pt}]
\node[main node,fill] at (2,2) (0) {};
\node[main node] at (3,1) (1) {};
\node[main node] at (4,0) (2) {};
\path
(0) edge node[pos=0.5, edge label, above right] {$2$} (1)
(1) edge node[pos=0.5, edge label, above right] {$3$} (2);
\end{tikzpicture}
&
\begin{tikzpicture}[xscale=\sc, yscale=\ysc,thick,main node/.style={circle, inner sep=0.5mm, draw, font=\small\sffamily},edge label/.style={font=\scriptsize,  inner sep=0pt}]
\node[main node,fill] at (2,2) (0) {};
\node[main node] at (3,1) (1) {};
\node[main node] at (4,0) (2) {};
\path
(0) edge node[pos=0.5, edge label, above right] {$3$} (1)
(1) edge node[pos=0.5, edge label, above right] {$3$} (2);
\end{tikzpicture}
&
\begin{tikzpicture}[xscale=\sc, yscale=\ysc,thick,main node/.style={circle, inner sep=0.5mm, draw, font=\small\sffamily},edge label/.style={font=\scriptsize,  inner sep=0pt}]
\node[main node,fill] at (2,2) (0) {};
\node[main node] at (1,1) (1) {};
\node[main node] at (3,1) (2) {};
\node at (2,0) (3) {};
\path
(0) edge (1)
(0) edge node[pos=0.5, edge label, above right] {$1$} (2);
\end{tikzpicture}
&
\begin{tikzpicture}[xscale=\sc, yscale=\ysc,thick,main node/.style={circle, inner sep=0.5mm, draw, font=\small\sffamily},edge label/.style={font=\scriptsize,  inner sep=0pt}]
\node[main node,fill] at (2,2) (0) {};
\node[main node] at (1,1) (1) {};
\node[main node] at (3,1) (2) {};
\node at (2,0) (3) {};
\path
(0) edge (1)
(0) edge node[pos=0.5, edge label, above right] {$2$} (2);
\end{tikzpicture}
\\
\begin{tikzpicture}[xscale=\sc, yscale=\ysc,thick,main node/.style={circle, inner sep=0.5mm, draw, font=\small\sffamily},edge label/.style={font=\scriptsize,  inner sep=0pt}]
\node[main node, fill] at (2,2) (0) {};
\node[main node] at (1,1) (1) {};
\node[main node] at (3,1) (2) {};
\node at (2,0) (3) {};
\path
(0) edge (1)
(0) edge node[pos=0.5, edge label, above right] {$3$} (2);
\end{tikzpicture}
&

\begin{tikzpicture}[xscale=\sc, yscale=\ysc,thick,main node/.style={circle, inner sep=0.5mm, draw, font=\small\sffamily},edge label/.style={font=\scriptsize,  inner sep=0pt}]
\node[main node,fill] at (2,2) (0) {};
\node[main node] at (1,1) (1) {};
\node[main node] at (2,0) (2) {};
\path
(0) edge (1)
(1) edge node[pos=0.5, edge label, above right] {$1$} (2);
\end{tikzpicture}
&

\begin{tikzpicture}[xscale=\sc, yscale=\ysc,thick,main node/.style={circle, inner sep=0.5mm, draw, font=\small\sffamily},edge label/.style={font=\scriptsize,  inner sep=0pt}]
\node[main node,fill] at (2,2) (0) {};
\node[main node] at (1,1) (1) {};
\node[main node] at (2,0) (2) {};
\path
(0) edge (1)
(1) edge node[pos=0.5, edge label, above right] {$2$} (2);
\end{tikzpicture}
&

\begin{tikzpicture}[xscale=\sc, yscale=\ysc,thick,main node/.style={circle, inner sep=0.5mm, draw, font=\small\sffamily},edge label/.style={font=\scriptsize,  inner sep=0pt}]
\node[main node,fill] at (2,2) (0) {};
\node[main node] at (1,1) (1) {};
\node[main node] at (2,0) (2) {};
\path
(0) edge (1)
(1) edge node[pos=0.5, edge label, above right] {$3$} (2);
\end{tikzpicture}
&
\begin{tikzpicture}[xscale=\sc, yscale=\ysc,thick,main node/.style={circle, inner sep=0.5mm, draw, font=\small\sffamily},edge label/.style={font=\scriptsize,  inner sep=0pt}]
\node[main node,fill] at (2,2) (0) {};
\node[main node] at (3,1) (1) {};
\node[main node] at (2,0) (2) {};
\path
(1) edge (2)
(0) edge node[pos=0.5, edge label, above right] {$1$} (1);
\end{tikzpicture}
&
\begin{tikzpicture}[xscale=\sc, yscale=\ysc,thick,main node/.style={circle, inner sep=0.5mm, draw, font=\small\sffamily},edge label/.style={font=\scriptsize,  inner sep=0pt}]
\node[main node,fill] at (2,2) (0) {};
\node[main node] at (3,1) (1) {};
\node[main node] at (2,0) (2) {};
\path
(1) edge (2)
(0) edge node[pos=0.5, edge label, above right] {$2$} (1);
\end{tikzpicture}
&
\begin{tikzpicture}[xscale=\sc, yscale=\ysc,thick,main node/.style={circle, inner sep=0.5mm, draw, font=\small\sffamily},edge label/.style={font=\scriptsize,  inner sep=0pt}]
\node[main node,fill] at (2,2) (0) {};
\node[main node] at (3,1) (1) {};
\node[main node] at (2,0) (2) {};
\path
(1) edge (2)
(0) edge node[pos=0.5, edge label, above right] {$3$} (1);
\end{tikzpicture}
&
\begin{tikzpicture}[xscale=\sc, yscale=\ysc,thick,main node/.style={circle, inner sep=0.5mm, draw, font=\small\sffamily},edge label/.style={font=\scriptsize,  inner sep=0pt}]
\node[main node,fill] at (2,2) (0) {};
\node[main node] at (1,1) (1) {};
\node[main node] at (0,0) (2) {};
\path
(0) edge (1)
(1) edge (2);
\end{tikzpicture}
\end{tabular}
\caption{The 16 rooted ordered binary trees counted by $b_{3,2}^c(6)$.
Each right edge is labeled by an element of $\{1,2,3\}$, and the labels are weakly
increasing along every maximal right-edge chain whose intermediate vertices have no left child.}\label{fig:binarytrees}
\end{figure}

We remark that, when $d=1$, all of the weak monotonicity conditions in
Propositions~\ref{prop:sec33trees}, \ref{prop:sec33pathsgeneral}, and
\ref{prop:sec33binary} are vacuous.  Hence the commuting and noncommuting models
coincide in this case, as expected.

\section{The commutative multiplication extension}\label{sec4}

We now turn to the case in which the binary multiplication of monomials is
associative \emph{and commutative}. (The unary operators may or may not commute.)
In the previous two sections, where multiplication was noncommutative, every monomial factored uniquely as an ordered sequence of atoms. In the present setting, multiplication is commutative, so each monomial factors uniquely as a multiset of atoms. Accordingly, the geometric-series relation
\[
A=\frac{\overline A}{1-\overline A}
\]
from Sections~\ref{sec2} and~\ref{sec3} is replaced by the standard multiset, or
exp--log, construction from unlabeled combinatorics.

As in the earlier sections, we keep $\ast$ as the abstract indeterminate.
To distinguish the present variant from the noncommutative-multiplication cases, we use a superscript $m$ when multiplication is commutative and the unary operators are noncommutative, and a superscript $cm$ when the unary operators also commute.

\subsection{Noncommuting unary operators}\label{sec41}

We first assume that the unary operators $P_1,\ldots,P_d$ do \emph{not} commute with one another, but that the binary multiplication is commutative. Let
\[
a_d^m(r;\mathbf{s})
\]
denote the number of inequivalent monomials in $V_d$ with degree $r$ and operator
multiplicity vector $\mathbf{s}=(s_1,\ldots,s_d)$. For example,  in Figure~\ref{fig:a221boxes}, monomials in the same row within the same box are equivalent when multiplication is commutative. Then counting the total number of rows within boxes gives $a_2^m(2;(2,1)) = 17$. Next, define
\[
A_d^m(u; \mathbf{p})
=
\sum_{r\ge1}\;\sum_{\mathbf{s}\ge\mathbf{0}}
a_d^m(r;\mathbf{s})\,u^r p_1^{s_1}\cdots p_d^{s_d}.
\]
Likewise, let $\overline a_d^m(r;\mathbf{s})$ and $\overline A_d^m(u;\mathbf{p})$ denote the corresponding counts and generating function for atoms. Exactly as in Section~\ref{sec2}, an atom is either the single indeterminate $\ast$ or a
monomial of the form $P_i(v)$. Hence
\begin{equation}\label{eq:sec4ncatom}
\overline A_d^m(u;\mathbf{p})
=u+(p_1+\cdots+p_d)A_d^m(u;\mathbf{p}).
\end{equation}

We next express $A_d^m$ in terms of $\overline A_d^m$. The essential new feature is that, because the product is now commutative, every monomial factors uniquely as a multiset of atoms. This places the problem in the standard exp--log framework for unlabeled multiset constructions. We briefly sketch the derivation; for a fuller treatment of this construction, see Flajolet and Sedgewick~\cite[Section~VII.2]{FlajoletS09}.

Let $v_1,v_2,\ldots$ be the list of all atoms in $V_d$. Since multiplication is
commutative, every nontrivial monomial $v\in V_d$ can be written uniquely in the form
\[
v=v_1^{t_1}v_2^{t_2}\cdots,
\]
where $t_1,t_2,\ldots\in\mathbb Z_{\ge0}$ and not all $t_i$ are zero. If an atom $v_i$
has degree $r$ and multiplicity vector $\mathbf{s}$, then its contribution to $A_d^m$ is
\[
\sum_{t_i\ge0}(u^r p_1^{s_1}\cdots p_d^{s_d})^{t_i}
=\frac{1}{1-u^r p_1^{s_1}\cdots p_d^{s_d}}.
\]
Multiplying over all atom types therefore gives
\begin{equation}\label{eq:vmultisetproduct}
1+A_d^m(u;\mathbf{p})
=
\prod_{r\ge1}\;\prod_{\mathbf{s}\ge \mathbf{0}}
\left(1-u^r p_1^{s_1} \cdots p_d^{s_d}\right)^{-\overline a_d^m(r;\mathbf{s})}.
\end{equation}
Taking the logarithm of the right-hand side and using the expansion $-\log(1-x)=\sum_{j\ge1} \frac{x^j}{j}$, we obtain the equivalent form
\begin{equation}\label{eq:vmultisetexp}
1+A_d^m(u;\mathbf{p})
=
\exp\!\left(
\sum_{j\ge1}
\frac{
\overline A_d^m(u^j;p_1^j,\ldots,p_d^j)
}{j}
\right).
\end{equation}
Substituting~\eqref{eq:sec4ncatom} into~\eqref{eq:vmultisetexp} gives the following master equation
\begin{equation}\label{eq:sec4ncmaster}
1+A_d^m(u;\mathbf{p})
=
\exp\!\left(
\sum_{j\ge1}
\frac{u^j+(p_1^j+\cdots+p_d^j)A_d^m(u^j;p_1^j,\ldots,p_d^j)}{j}
\right).
\end{equation}

As in the previous sections, the multivariate functional equation yields an effective
coefficient recurrence. By \eqref{eq:sec4ncatom}, we have
\[
\overline a_d^m(r;\mathbf{s})
=
\delta_{r,1}\,\delta_{\mathbf{s},\mathbf{0}}
+
\sum_{i=1}^d a_d^m(r;\mathbf{s}-\mathbf e_i),
\]
where $\mathbf e_i$ is the $i$th standard basis vector and, as usual, terms with negative
coordinates are interpreted as $0$.

Applying the Euler-transform argument gives the following multigraded recurrence.

\begin{theorem}\label{thm:sec4ncmultirec}
Let $a_d^m(0;\mathbf{0})=1$ and $a_d^m(0;\mathbf{s})=0$ for $\mathbf{s}\neq\mathbf{0}$.
For $r\ge1$ and $\mathbf{s}\in\mathbb{Z}_{\ge0}^d$, define
\[
c_d^m(r;\mathbf{s})
:=
\sum_{\substack{j \mid \gcd(r,s_1,\ldots,s_d)}} 
\frac{r}{j} \overline a_d^m\!\left(\frac{r}{j};\frac{\mathbf{s}}{j}\right),
\]
where the sum ranges over all positive integers $j$ dividing every coordinate of
$(r,s_1,\ldots,s_d)$.
Then
\begin{equation}\label{eq:sec4ncmultirec}
r\,a_d^m(r;\mathbf{s})
=
\sum_{\substack{1\le j\le r\\ \mathbf{0} \le \boldsymbol{\alpha}\le\mathbf{s}}}
c_d^m(j;\boldsymbol{\alpha})\,
a_d^m(r-j;\mathbf{s}-\boldsymbol{\alpha}),
\end{equation}
with all vector inequalities interpreted coordinatewise.
\end{theorem}

\begin{proof}
For brevity, we write $A(u;\mathbf{p}) := A_d^m(u;\mathbf{p})$ throughout the proof. From the multiset product formula~\eqref{eq:vmultisetproduct}, we have
\begin{equation}\label{eq:sec4ncproofprod}
1+A(u;\mathbf{p})
=
\prod_{r\ge1}\;\prod_{\mathbf{s}\ge\mathbf{0}}
\left(1-u^r p_1^{s_1}\cdots p_d^{s_d}\right)^{-\overline a_d^m(r;\mathbf{s})},
\end{equation}
Taking logarithms of both sides of \eqref{eq:sec4ncproofprod}, we obtain
\[
\log(1+A(u;\mathbf{p}))
=
-\sum_{r\ge1}\;\sum_{\mathbf{s}\ge\mathbf{0}}
\overline a_d^m(r;\mathbf{s})
\log\!\left(1-u^r p_1^{s_1}\cdots p_d^{s_d}\right).
\]
Now differentiate with respect to $u$ and multiply by $u$.  This gives
\begin{align}
\frac{u\,\partial A/\partial u}{1+A(u;\mathbf{p})}
&=
\sum_{r\ge1}\;\sum_{\mathbf{s}\ge\mathbf{0}}
\overline a_d^m(r;\mathbf{s})\,
\frac{r\,u^r p_1^{s_1}\cdots p_d^{s_d}}
{1-u^r p_1^{s_1}\cdots p_d^{s_d}} \notag\\
&=
\sum_{r\ge1}\;\sum_{\mathbf{s}\ge\mathbf{0}}
\overline a_d^m(r;\mathbf{s})
\sum_{j\ge1}
r\,u^{jr}p_1^{js_1}\cdots p_d^{js_d},
\label{eq:sec4ncprooflogder}
\end{align}
where in the second step we used the geometric-series identity 
$
\frac{x}{1-x}=\sum_{j\ge1}x^j.
$
Next, we collect the coefficient of a fixed monomial
\[
u^n p_1^{\beta_1}\cdots p_d^{\beta_d}.
\]
Such a monomial appears on the right-hand side of \eqref{eq:sec4ncprooflogder}
precisely when
\[
n=jr,
\qquad
\beta_i=js_i \quad (1\le i\le d),
\]
that is, when $j$ divides every coordinate of $(n,\beta_1,\ldots,\beta_d)$.
Hence we may rewrite \eqref{eq:sec4ncprooflogder} as
\[
\frac{u\,\partial A/\partial u}{1+A(u;\mathbf{p})}
=
\sum_{n\ge1}\;\sum_{\boldsymbol{\beta}\ge\mathbf{0}}
c_d^m(n;\boldsymbol{\beta})\,
u^n p_1^{\beta_1}\cdots p_d^{\beta_d},
\]
where
\[
c_d^m(n;\boldsymbol{\beta})
=
\sum_{j\mid\gcd(n,\beta_1,\ldots,\beta_d)}
\frac{n}{j}\,
\overline a_d^m\!\left(\frac{n}{j};\frac{\boldsymbol{\beta}}{j}\right).
\]
Multiplying both sides by $1+A(u;\mathbf{p})$, we obtain
\[
u\,\frac{\partial A}{\partial u}
=
\bigl(1+A(u;\mathbf{p})\bigr)
\sum_{n\ge1}\;\sum_{\boldsymbol{\beta}\ge\mathbf{0}}
c_d^m(n;\boldsymbol{\beta})\,
u^n p_1^{\beta_1}\cdots p_d^{\beta_d}.
\]

Finally, expand both sides in coefficients.  By definition,
\[
A(u;\mathbf{p})
=
\sum_{r\ge1}\;\sum_{\mathbf{s}\ge\mathbf{0}}
a_d^m(r;\mathbf{s})\,u^r p_1^{s_1}\cdots p_d^{s_d},
\]
and therefore
\[
u\,\frac{\partial A}{\partial u}
=
\sum_{r\ge1}\;\sum_{\mathbf{s}\ge\mathbf{0}}
r\,a_d^m(r;\mathbf{s})\,u^r p_1^{s_1}\cdots p_d^{s_d}.
\]
Also, using the initial conditions
\[
a_d^m(0;\mathbf{0})=1,
\qquad
a_d^m(0;\mathbf{s})=0 \quad (\mathbf{s}\neq\mathbf{0}),
\]
we may write
\[
1+A(u;\mathbf{p})
=
\sum_{r\ge0}\;\sum_{\mathbf{s}\ge\mathbf{0}}
a_d^m(r;\mathbf{s})\,u^r p_1^{s_1}\cdots p_d^{s_d}.
\]
Comparing the coefficient of
$u^r p_1^{s_1}\cdots p_d^{s_d}$ on both sides now gives
\[
r\,a_d^m(r;\mathbf{s})
=
\sum_{\substack{1\le j\le r\\ \mathbf{0}\le\boldsymbol{\alpha}\le\mathbf{s}}}
c_d^m(j;\boldsymbol{\alpha})\,
a_d^m(r-j;\mathbf{s}-\boldsymbol{\alpha}),
\]
where all vector inequalities are interpreted coordinatewise.  This is exactly
\eqref{eq:sec4ncmultirec}.
\end{proof}

\begin{table}[ht]
\centering
\begin{tabular}{c|l|l}
$r$ & First terms of $a_1^m(r;k)$ & OEIS entry \\
\hline
\\[-12pt]
1& $1,1,1,1,1,1,1,\ldots $& \seqnum{A000012}\\
2& $1,2,4,6,9,12,16,20,25, 30, 36, 42, 49, 56, 64, 72, 81, 90, 100, 110, 121,\ldots $& \seqnum{A002620}\\
3& $1,3,8,18,35,62,103,161,241, 348, 487, 664, 886, 1159, 1491, 1890,\ldots $& \seqnum{A055278}\\
4& $1,4,14,39,97,212,429,804,1427, 2406, 3900, 6094, 9245, 13645,\ldots $& \seqnum{A055279}\\
5& $1,5,21,72,214,563,1344,2958,6086, 11820, 21854, 38713, 66069,\ldots $& \seqnum{A055280}
\end{tabular}
\caption{First terms of the sequences $a_1^m(r;k)$ for $r = 1,2,3,4,5$.}
\label{tab:a1mr}
\end{table}

Table~\ref{tab:a1mr} lists the initial terms of the sequences $a_1^m(r;k)$ for $r \in \{1,2,3,4,5\}$. 
In all these cases, we obtain known sequences corresponding to rooted unordered trees with a prescribed number of leaves. Thus, in the commutative-multiplication setting, the one-operator refinement $a_1^m(r;k)$ is no longer governed by the Narayana numbers, but instead gives the classical rooted-tree-by-leaves triangle (\seqnum{A055277}). 

This rooted-tree interpretation is also natural from the tree model of
Section~\ref{sec2}. There, multiplication is represented by an internal node whose
children record the factors from left to right. Once multiplication is made commutative,
these children become unordered. Thus the rooted ordered trees of the noncommutative setting collapse to rooted unordered trees, which explains the appearance of the rooted-tree sequences in the case of $d=1$. For general $d$, we obtain operator-colored analogs of rooted unordered trees.

As before, we define from $a_d^m(r;\mathbf{s})$ length-graded sequences
\[
b_{d,\ell}^m(n) := 
\sum_{\substack{r\ge1,\;\mathbf{s}\ge\mathbf{0}\\ \ell r+2|\mathbf{s}|=n}}
a_d^m(r;\mathbf{s})
\]
for every $d, \ell \geq 1$. Then the generating function $B_{d,\ell}^m(z):=\sum_{n\ge1} b_{d,\ell}^m(n)z^n$ satisfies
\[
B_{d,\ell}^m(z)=A_d^m(z^{\ell} ;z^2,\ldots,z^2).
\]
Specializing the multiset master equation~\eqref{eq:sec4ncmaster} yields an effective
one-variable recurrence for the length-graded sequences.

\begin{theorem}\label{thm:sec4bdlmrec}
For $n\ge1$, define
\[
\overline b_{d,\ell}^m(n)
:=
\delta_{n,\ell}
+
d\,b_{d,\ell}^m(n-2),
\]
with the convention $b_{d,\ell}^m(j)=0$ for $j\le0$. Next define
\[
c_{d,\ell}^m(n)
:=
\sum_{k\mid n} k\,\overline b_{d,\ell}^m(k).
\]
Then, for every $n\ge1$,
\begin{equation}\label{eq:sec4bdlmrec}
n b_{d,\ell}^m(n)
=
c_{d,\ell}^m(n)
+
\sum_{k=1}^{n-1} c_{d,\ell}^m(k)\,b_{d,\ell}^m(n-k).
\end{equation}
\end{theorem}

\begin{proof}
Define $\overline B_{d,\ell}^m(z) :=\sum_{n\ge1}\overline b_{d,\ell}^m(n)z^n$. Then
\[
\overline B_{d,\ell}^m(z):=z^\ell+d z^2 B_{d,\ell}^m(z).
\]
Specializing~\eqref{eq:sec4ncmaster} then gives
\[
1+B_{d,\ell}^m(z)
=
\exp\!\left(
\sum_{j\ge1}\frac{\overline B_{d,\ell}^m(z^j)}{j}
\right).
\]
Taking logarithms, differentiating with respect to $z$, and multiplying by $z$, we obtain
\[
\frac{z(B_{d,\ell}^m)'(z)}{1+B_{d,\ell}^m(z)}
=
\sum_{n\ge1}\left(\sum_{k\mid n} k\,\overline b_{d,\ell}^m(k)\right)z^n
=
\sum_{n\ge1} c_{d,\ell}^m(n)z^n.
\]
Multiplying both sides by $1+B_{d,\ell}^m(z)$ and comparing coefficients of $z^n$
gives~\eqref{eq:sec4bdlmrec}.
\end{proof}

Table~\ref{tab:bmd2} lists the initial terms of the sequences $b_{d,2}^m(2n)$ for $d \leq 4$.

\begin{table}[ht]
\centering
\begin{tabular}{c|l|l}
$d$ & First terms of $b_{d,2}^m(2n)$ & OEIS entry \\
\hline
\\[-12pt]
1 & $1,2,4,9,20,48,115,286,719,1842,4766, 12486, 32973, \ldots$ & \seqnum{A000081} \\
2 & $1,3,9,30,102,367,1347,5081,19491,75960,299622,\ldots$ & \seqnum{A394948}\textsuperscript{*}\\
3 & $1,4,16,70,316,1496,7262,36125,182892,939930,\ldots$ & \seqnum{A394949}\textsuperscript{*}\\
4 & $1,5,25,135,755,4405,26385,161730,1008870,6385736,\ldots$ & \seqnum{A394950}\textsuperscript{*}\\
\end{tabular}
\caption{First terms of the sequences $b_{d,2}^m(2n)$ for $d=1,2,3,4$.}
\label{tab:bmd2}
\end{table}

Unsurprisingly, we obtain the rooted-tree sequence \seqnum{A000081} (shifted
by one place) for $d=1$, which is exactly the row sums of the aforementioned array \seqnum{A055277}. For general $d$, the monomials again correspond naturally to rooted unordered trees in which unary nodes are assigned a label in $[d]$. 

\subsection{Commuting unary operators}\label{sec42}

We now consider the case in which the $d$ unary operators commute pairwise and the
binary multiplication is also commutative. Let
\[
a_d^{cm}(r;\mathbf{s})
\]
be the number of equivalence classes of monomials with degree $r$ and multiplicity vector $\mathbf{s}$ in this case. For example,  in Figure~\ref{fig:a221boxes}, monomials within the same box now all belong to the same equivalence class, which gives $a_2^{cm}(2;(2,1)) = 10$. Next, let
\[
A_d^{cm}(u;\mathbf{p})
:=
\sum_{r\ge1}\;\sum_{\mathbf{s}\ge\mathbf{0}}
a_d^{cm}(r;\mathbf{s})\,u^r p_1^{s_1}\cdots p_d^{s_d}
\]
denote the corresponding multigraded counting function. Because the unary operators commute, the same argument as in Lemma~\ref{lem:sec3atom} applies, and we obtain that
\begin{equation}\label{eq:sec4cmatom}
\overline A_d^{cm}(u;\mathbf{p})
=u+\left(1-\prod_{i=1}^d(1-p_i)\right)A_d^{cm}(u;\mathbf{p}).
\end{equation}
Since multiplication is also commutative in this case, the multiset-of-atoms argument that led to~\eqref{eq:vmultisetexp} also applies here, which implies
\begin{equation}\label{eq:vmultisetexpcm}
1+A_d^{cm}(u;\mathbf{p})
=
\exp\!\left(
\sum_{j\ge1}
\frac{
\overline A_d^{cm}(u^j;p_1^j,\ldots,p_d^j)
}{j}
\right).
\end{equation}
Combining~\eqref{eq:sec4cmatom} with~\eqref{eq:vmultisetexpcm}, we obtain
\begin{equation}\label{eq:sec4cmmaster}
1+A_d^{cm}(u;\mathbf{p})
=
\exp\!\left(
\sum_{j \ge1}
\frac{u^j+\left(1-\prod_{i=1}^d(1-p_i^j)\right)A_d^{cm}(u^j;p_1^j,\ldots,p_d^j)}{j}
\right).
\end{equation}
As in the previous subsection, this gives an effective recursive scheme for computing 
$a_d^{cm}(r;\mathbf{s})$.  Making this explicit, we obtain the following multigraded recurrence by the same logarithmic-derivative argument used in Theorem~\ref{thm:sec4ncmultirec}.

\begin{theorem}\label{thm:sec4cmmultirec}
Let $a_d^{cm}(0;\mathbf{0})=1$ and $a_d^{cm}(0;\mathbf{s})=0$ for
$\mathbf{s}\neq\mathbf{0}$. For $r\ge1$ and
$\mathbf{s}\in\mathbb{Z}_{\ge0}^d$, define
\[
\overline a_d^{cm}(r;\mathbf{s})
:=
\delta_{r,1}\,\delta_{\mathbf{s},\mathbf{0}}
+
\sum_{\emptyset\neq J\subseteq[d]}
(-1)^{|J|+1}
a_d^{cm}(r;\mathbf{s}-\mathbf e_J),
\]
where $\mathbf e_J\in\{0,1\}^d$ denotes the indicator vector of $J$, and terms with
negative coordinates are interpreted as $0$. Next define
\[
c_d^{cm}(r;\mathbf{s})
:=
\sum_{\substack{j\mid \gcd(r,s_1,\ldots,s_d)}}
\frac{r}{j}\,
\overline a_d^{cm}\!\left(\frac{r}{j};\frac{\mathbf{s}}{j}\right).
\]
Then
\begin{equation}\label{eq:sec4cmmultirec}
r\,a_d^{cm}(r;\mathbf{s})
=
\sum_{\substack{1\le j\le r\\ \mathbf{0}\le \boldsymbol{\alpha}\le \mathbf{s}}}
c_d^{cm}(j;\boldsymbol{\alpha})\,
a_d^{cm}(r-j;\mathbf{s}-\boldsymbol{\alpha}),
\end{equation}
where all vector inequalities are interpreted coordinatewise.
\end{theorem}

\begin{proof}
The proof is identical to that of Theorem~\ref{thm:sec4ncmultirec}.  Using
\eqref{eq:sec4cmatom} in place of~\eqref{eq:sec4ncatom}, we first obtain the
displayed formula for $\overline a_d^{cm}(r;\mathbf{s})$ by extracting coefficients
from
\[
\overline A_d^{cm}(u;\mathbf{p})
=
u+\left(1-\prod_{i=1}^d(1-p_i)\right)A_d^{cm}(u;\mathbf{p}).
\]
Applying the same logarithmic-derivative argument to the multiset identity
\eqref{eq:vmultisetexpcm} gives~\eqref{eq:sec4cmmultirec}.
\end{proof}

We can then define the analogous length-graded sequences for each $d, \ell \geq 1$:
\[
b_{d,\ell}^{cm}(n) := 
\sum_{\substack{r\ge1,\;\mathbf{s}\ge\mathbf{0}\\ \ell r+2|\mathbf{s}|=n}}
a_d^{cm}(r;\mathbf{s})
\]
The corresponding length-graded sequences also satisfy an effective Euler-transform
recurrence.

\begin{theorem}\label{thm:sec4bdlcmrec}
Fix $d,\ell\ge1$. For $n\ge1$, define
\[
\overline b_{d,\ell}^{cm}(n)
:=
\delta_{n,\ell}
+
\sum_{j=1}^{\min(d,\lfloor n/2\rfloor)}
(-1)^{j+1}\binom{d}{j}\,
b_{d,\ell}^{cm}(n-2j),
\]
with the convention $b_{d,\ell}^{cm}(m)=0$ for $m\le0$. Next define
\[
c_{d,\ell}^{cm}(n)
:=
\sum_{k\mid n} k\,\overline b_{d,\ell}^{cm}(k).
\]
Then, for every $n\ge1$,
\begin{equation}\label{eq:sec4bdlcmrec}
n\,b_{d,\ell}^{cm}(n)
=
c_{d,\ell}^{cm}(n)
+
\sum_{k=1}^{n-1} c_{d,\ell}^{cm}(k)\,b_{d,\ell}^{cm}(n-k).
\end{equation}
\end{theorem}

\begin{proof}
Let $B_{d,\ell}^{cm}(z) :=\sum_{n\ge1} b_{d,\ell}^{cm}(n)z^n$ and $\overline B_{d,\ell}^{cm}(z) := \sum_{n\ge1} \overline b_{d,\ell}^{cm}(n)z^n$. Then
\[
\overline B_{d,\ell}^{cm}(z)
=
z^\ell+\bigl(1-(1-z^2)^d\bigr)B_{d,\ell}^{cm}(z).
\]
Then specializing~\eqref{eq:sec4cmmaster} gives
\[
1+B_{d,\ell}^{cm}(z)
=
\exp\!\left(
\sum_{j\ge1}\frac{\overline B_{d,\ell}^{cm}(z^j)}{j}
\right).
\]
Taking logarithms, differentiating with respect to $z$, and multiplying by $z$, we obtain
\[
\frac{z(B_{d,\ell}^{cm})'(z)}{1+B_{d,\ell}^{cm}(z)}
=
\sum_{n\ge1}\left(\sum_{k\mid n} k\,\overline b_{d,\ell}^{cm}(k)\right)z^n
=
\sum_{n\ge1} c_{d,\ell}^{cm}(n)z^n.
\]
Multiplying both sides by $1+B_{d,\ell}^{cm}(z)$ and comparing coefficients of $z^n$
gives~\eqref{eq:sec4bdlcmrec}.
\end{proof}

\begin{table}[ht]
\centering
\begin{tabular}{c|l|l}
$d$ & First terms of $b_{d,2}^{cm}(2n)$ & OEIS entry \\
\hline
\\[-12pt]
1 & $1,2,4,9,20,48,115,286,719,1842,4766, 12486, 32973, \ldots$ & \seqnum{A000081} \\
2 & $1,3,8,24,74,243,815,2815,9899,35400,128203,469532,\ldots$ & \seqnum{A394951}\textsuperscript{*} \\
3 & $1,4,13,47,180,731,3055,13121,57473,255930,1154703,\ldots$ & \seqnum{A394952}\textsuperscript{*}\\
4 & $1,5,19,79,356,1702,8395,42516,219732,1154483,\ldots$ & \seqnum{A394953}\textsuperscript{*}
\end{tabular}
\caption{First terms of the sequences $b_{d,2}^{cm}(2n)$ for $d=1,2,3,4$.}
\label{tab:bcmd2}
\end{table}

Notice that
\[
b_{1,\ell}^m(n) = b_{1,\ell}^{cm}(n)
\]
for every $\ell, n \geq 1$. Thus, as shown in Table~\ref{tab:bcmd2}, $b_{1,2}^{cm}(2n)$ 
again reduces to the rooted-tree sequence \seqnum{A000081}. Also, as with $b_{d,2}^m(2n)$, $b_{d,2}^{cm}(2n)$ also counts rooted unordered trees where each unary node has one of $d$ colors. However, while the noncommuting unary case allows arbitrary colorings of the unary nodes, the commuting unary case imposes monotonicity restrictions on the labels as described in Section~\ref{sec33}.

This completes the case-by-case enumeration in the commutative-multiplication regimes. In the final section, we compare the asymptotic growth of the four length-graded families arising from the noncommuting and commuting versions of both the unary operators and the multiplication.

\section{Comparison of the four length-graded families}\label{sec5}

We conclude by comparing the asymptotic behavior of the four length-graded families
\[
b_{d,\ell}(n), \qquad b_{d,\ell}^c(n), \qquad b_{d,\ell}^m(n), \qquad b_{d,\ell}^{cm}(n).
\]
These correspond to the four commutativity regimes studied in Sections~\ref{sec2}--\ref{sec4}. For every $d,\ell\ge1$, define
\[
g_{d,\ell}:=\lim_{n\to\infty}\sqrt{\frac{b_{d,\ell}(2n+2)}{b_{d,\ell}(2n)}}.
\]
We note that this particular definition of the exponential growth rate of $b_{d,\ell}(n)$ allows us to treat the even-support cases for even $\ell$ and the full-support cases for odd $\ell$ in a uniform way. We also define the quantities
\[
g_{d,\ell}^c, \quad
g_{d,\ell}^m, \quad
g_{d,\ell}^{cm}
\]
analogously whenever the corresponding limits exist. In the noncommutative-multiplication settings, the algebraic equations for $B_{d,\ell}(z)$ and $B_{d,\ell}^c(z)$ allow us to determine these growth rates analytically. We record these values first.

\begin{proposition}\label{prop:sec5:gdl}
For every $d,\ell \ge 1$, we have:
\begin{itemize}
\item[\textup{(i)}]
$g_{d,\ell}=\rho_{d,\ell}^{-1}$, where $\rho_{d,\ell}$ is the unique solution in $(0,1)$ of
\[
\rho^{\ell/2}+\sqrt d\,\rho=1.
\]

\item[\textup{(ii)}]
$g_{d,\ell}^c=(\rho_{d,\ell}^c)^{-1}$, where $\rho_{d,\ell}^c$ is the unique solution in $(0,1)$ of
\[
(1-\rho^2)^d+\rho^\ell=2\rho^{\ell/2}.
\]
\end{itemize}
\end{proposition}

\begin{proof}
We first consider $B_{d,\ell}(z)$. By Theorem~\ref{thm:sec2bdlgeneral},
\[
B_{d,\ell}(z)
=
\frac{1-z^\ell-dz^2-\sqrt{\Delta_{d,\ell}(z)}}{2dz^2},
\]
where the discriminant is
\begin{align*}
\Delta_{d,\ell}(z)
&=
(1-z^\ell-dz^2)^2-4dz^{\ell+2}\\
&=
\bigl(1-z^{\ell/2}-\sqrt d\,z\bigr)
\bigl(1-z^{\ell/2}+\sqrt d\,z\bigr)
\bigl(1+z^{\ell/2}-\sqrt d\,z\bigr)
\bigl(1+z^{\ell/2}+\sqrt d\,z\bigr).
\end{align*}
Hence, the smallest positive zero of $\Delta_{d,\ell}(z)$ is the unique solution
$\rho_{d,\ell}\in(0,1)$ of
\[
z^{\ell/2}+\sqrt d\,z=1,
\]
since the left-hand side is strictly increasing on $(0,\infty)$.
Because $B_{d,\ell}(z)$ has nonnegative coefficients, Pringsheim's theorem (see, for instance, Flajolet and Sedgewick~\cite[Chapter~IV]{FlajoletS09}) implies that the singularity of smallest modulus must occur at a positive real value of $z$, so the smallest positive zero $\rho_{d,\ell}$ of the discriminant is the dominant singularity (that is, the singularity closest to the origin), and hence determines the radius of convergence.
Moreover, this zero of the discriminant is simple, so $B_{d,\ell}(z)$ has a square-root
singularity at $z=\rho_{d,\ell}$. It follows from the standard transfer theorem for
square-root singularities that
\[
b_{d,\ell}(n)\sim C_{d,\ell}\,\rho_{d,\ell}^{-n}n^{-3/2}
\]
when $\ell$ is odd, while for even $\ell$ the same conclusion holds for the even
subsequence after writing $B_{d,\ell}(z)=\widetilde B_{d,\ell}(z^2)$; see
Flajolet and Sedgewick~\cite[Chapter~VI]{FlajoletS09}. In either case, we have
\[
g_{d,\ell} = \lim_{n \to \infty} \sqrt{\frac{b_{d,\ell}(2n+2)}{b_{d,\ell}(2n)}} = \rho_{d,\ell}^{-1}.
\]
Next consider $B_{d,\ell}^c(z)$. By Theorem~\ref{thm:sec3bdlcgeneral}, it satisfies a quadratic equation whose discriminant is
\[
\Delta_{d,\ell}^c(z)
=
\bigl((1-z^2)^d+z^\ell\bigr)^2-4z^\ell.
\]
Thus the positive zeros of $\Delta_{d,\ell}^c(z)$ are exactly the positive solutions of
\[
(1-z^2)^d+z^\ell=2z^{\ell/2}.
\]
Let
\[
F_{d,\ell}(z):=(1-z^2)^d+z^\ell-2z^{\ell/2}.
\]
Then $F_{d,\ell}(0)=1$ and $F_{d,\ell}(1)=-1$, so $F_{d,\ell}$ has a zero in $(0,1)$.
Also,
\[
F_{d,\ell}'(z)
=
-2dz(1-z^2)^{d-1}
-\ell z^{\ell/2-1}(1-z^{\ell/2})<0
\qquad (0<z<1),
\]
so this zero is unique; denote it by $\rho_{d,\ell}^c$.
Again, by Pringsheim's theorem, $\rho_{d,\ell}^c$ is the dominant singularity of
$B_{d,\ell}^c(z)$. Since it is a simple zero of the discriminant, the same square-root transfer theorem yields
\[
b_{d,\ell}^c(n)\sim C_{d,\ell}^c\,(\rho_{d,\ell}^c)^{-n}n^{-3/2}
\]
when $\ell$ is odd, and the corresponding even-subsequence asymptotic when $\ell$ is even. Therefore, 
\[
g_{d,\ell}^c = \lim_{n \to \infty} \sqrt{\frac{b_{d,\ell}^c(2n+2)}{b_{d,\ell}^c(2n)}}
=
(\rho_{d,\ell}^c)^{-1}.
\]
\end{proof}

Figure~\ref{fig:gdl} compares the resulting exponential growth rates.
The values $g_{d,\ell}$ and $g_{d,\ell}^c$ are determined analytically from
Proposition~\ref{prop:sec5:gdl}. In the commutative-multiplication settings, we do not presently have comparably explicit formulas, so in Figure~\ref{fig:gdl} we plot numerical estimates, denoted by $\hat g_{d,\ell}^m$ and $\hat g_{d,\ell}^{cm}$. More precisely, these estimates are computed by evaluating 
\[
\sqrt{\frac{b(2n+2)}{b(2n)}}\left(\frac{n+1}{n}\right)^{3/4}
\]
at values of $n$ around 100, using the expected $n^{-3/2}$ subexponential correction. The plotted values are rounded to four decimal places.

\def\d{8}
\def\xlb{0}
\def\xub{\d}
\def\ylb{0}
\def\yub{5}
\def\xbuf{0.5}
\def\ybuf{0.2}

\newcommand{\drawseries}[3]{
\draw[thick, mark=*, mark options={fill}, mark size=1pt] plot coordinates {#1};
     \node[anchor=west] at #3 {#2};
}

\begin{figure}[htb]
\centering
\begin{tabular}{cccc}

\begin{tikzpicture}[scale=0.5, xscale=0.55, yscale=1.6, font=\tiny]

\def\dataone{(1,2.618) (2,3.2043) (3,3.6399) (4,4) (5,4.3128) (6,4.5925) (7,4.8474) (8,5.083)}\def\datatwo{(1,2) (2,2.4142) (3,2.7321) (4,3) (5,3.2361) (6,3.4495) (7,3.6458) (8,3.8284)}\def\datathree{(1,1.7549) (2,2.1037) (3,2.3802) (4,2.618) (5,2.8305) (6,3.0245) (7,3.2044) (8,3.3729)}
\def\datafour{(
1,1.618) (2,1.9319) (3,2.1889) (4,2.4142) (5,2.618) (6,2.8059) (7,2.9812) (8,3.1463
)}
\def\datafive{(
 1,1.5289) ( 2, 1.8211) ( 3, 2.0682) ( 4, 2.2888) ( 5, 2.4905) ( 6, 2.6777) (7, 2.8532) ( 8, 3.0191
)}
\def\dataten{(
 1 ,1.3247) ( 2, 1.5762) ( 3, 1.8227) ( 4, 2.0560) ( 5, 2.2735) ( 6, 2.4761) (
 7, 2.6656) ( 8 ,2.8437
 )}

\draw[->] (0,0) -- (\xub+\xbuf,0);
\draw[->] (0,0) -- (0,\yub+\ybuf);

\foreach \x in {\xlb,1,...,\xub}{
    \ifthenelse{\NOT 0 = \x}{
        \draw[thick] (\x,-2pt) -- (\x,2pt);
        \node[anchor=north] at (\x,0) {$\x$};
    }{}
}

\foreach \y in {1,2,3,4,5}{
    \draw[thick] (-3pt,\y) -- (3pt,\y);
    \node[anchor=east] at (0,\y) {$\y$};
}

\node[anchor=north east] at (0,0) {$0$};
\node[anchor=west] at (\xub+\xbuf,0) {$d$};
\node[anchor=south] at (0,\yub+\ybuf) {\footnotesize $ g_{d,\ell}$};

\drawseries{\dataone}{$\ell=1$}{(8,5.083)}
\drawseries{\datatwo}{$\ell=2$}{(8,3.8284)}
\drawseries{\datathree}{$\ell=3$}{(8,3.3729)}
\drawseries{\datafive}{$\ell=5$}{(8,3.0291)}
\drawseries{\dataten}{$\ell=10$}{(8,2.8337)}

\end{tikzpicture}
&
\begin{tikzpicture}[scale=0.5, xscale=0.55, yscale=1.6, font=\tiny]

\def\dataone{(
  1, 2.618) ( 2, 3.1542) ( 3, 3.5471) ( 4, 3.8686) (5, 4.1459) ( 6, 4.3924) ( 7, 4.6158) ( 8, 4.8213)}
  \def\datatwo{(1 ,2) (2 ,2.3486) ( 3, 2.6062) ( 4, 2.8182) ( 5, 3.0017) ( 6, 3.1651) ( 7, 3.3136) ( 8, 3.4504)}
  \def\datathree{( 1 ,1.7549) (2 ,2.0277) ( 3, 2.2308) ( 4, 2.3986) ( 5, 2.5442) ( 6, 2.6742) ( 7, 2.7926) ( 8, 2.9017)}
\def\datafour{( 1 ,1.618) ( 2, 1.8478) ( 3, 2.0198) ( 4, 2.1625) ( 5, 2.2865) ( 6, 2.3975) ( 7, 2.4987) ( 8, 2.5921)}
\def\datafive{(
  1 ,1.5289) ( 2 ,1.7301) ( 3, 1.8816) ( 4, 2.0075) ( 5, 2.1173) ( 6, 2.2157) ( 7, 2.3055) ( 8, 2.3884
)}
\def\dataten{(
 1 ,1.3247) ( 2, 1.4580) ( 3, 1.5603) ( 4, 1.6463) ( 5, 1.7218) ( 6, 1.7899) ( 7, 1.8522) ( 8, 1.9101
 )}

\draw[->] (0,0) -- (\xub+\xbuf,0);
\draw[->] (0,0) -- (0,\yub+\ybuf);

\foreach \x in {\xlb,1,...,\xub}{
    \ifthenelse{\NOT 0 = \x}{
        \draw[thick] (\x,-2pt) -- (\x,2pt);
        \node[anchor=north] at (\x,0) {$\x$};
    }{}
}

\foreach \y in {1,2,3,4,5}{
    \draw[thick] (-3pt,\y) -- (3pt,\y);
    \node[anchor=east] at (0,\y) {$\y$};
}

\node[anchor=north east] at (0,0) {$0$};
\node[anchor=west] at (\xub+\xbuf,0) {$d$};
\node[anchor=south] at (0,\yub+\ybuf) {\footnotesize $ g_{d,\ell}^c$};

\drawseries{\dataone}{$\ell=1$}{(8,4.8213)}
\drawseries{\datatwo}{$\ell=2$}{(8,3.4504)}
\drawseries{\datathree}{$\ell=3$}{(8,2.9017)}
\drawseries{\datafive}{$\ell=5$}{(8,2.3884)}
\drawseries{\dataten}{$\ell=10$}{(8,1.9101)}

\end{tikzpicture}

&
\begin{tikzpicture}[scale=0.5, xscale=0.55, yscale=1.6, font=\tiny]

\def\dataone{(
1,2.0547) ( 2, 2.57) ( 3, 2.9615) ( 4 ,3.2886) ( 5, 3.5748) ( 6, 3.8319) ( 7, 4.0672) ( 8, 4.2853
)}\def\datatwo{(
 1 ,1.7194) ( 2 ,2.1201) ( 3 ,2.4332) ( 4, 2.699) ( 5, 2.9339) ( 6, 3.1467) ( 7, 3.3426) ( 8, 3.5251
)}\def\datathree{(
 1 ,1.5692) ( 2, 1.9203) ( 3, 2.2025) ( 4, 2.4459) ( 5, 2.6634) ( 6, 2.8619) ( 7, 3.0457) ( 8, 3.2178
)}\def\datafour{(
 1 ,1.4798) ( 2, 1.8024) ( 3, 2.0694) ( 4, 2.3032) ( 5, 2.5142) ( 6, 2.708) ( 7, 2.8884) ( 8 ,3.0579
)}
\def\datafive{(
1,1.4190) (2,1.7233) (3,1.9825) (4,2.2128) (5,2.4222) (6,2.6157) (7,2.7965) (8,2.9668
)}
\def\dataten{(
1,1.2706) (2,1.5393) (3,1.7997) (4,2.0419) (5,2.2661) (6,2.4745) (7,2.6692) (8,2.8518
 )}

\draw[->] (0,0) -- (\xub+\xbuf,0);
\draw[->] (0,0) -- (0,\yub+\ybuf);

\foreach \x in {\xlb,1,...,\xub}{
    \ifthenelse{\NOT 0 = \x}{
        \draw[thick] (\x,-2pt) -- (\x,2pt);
        \node[anchor=north] at (\x,0) {$\x$};
    }{}
}

\foreach \y in {1,2,3,4,5}{
    \draw[thick] (-3pt,\y) -- (3pt,\y);
    \node[anchor=east] at (0,\y) {$\y$};
}

\node[anchor=north east] at (0,0) {$0$};
\node[anchor=west] at (\xub+\xbuf,0) {$d$};
\node[anchor=south] at (0,\yub+\ybuf) {\footnotesize $ \hat{g}_{d,\ell}^m$};

\drawseries{\dataone}{$\ell=1$}{(8,4.2853)}
\drawseries{\datatwo}{$\ell=2$}{(8,3.5251)}
\drawseries{\datathree}{$\ell=3$}{(8,3.2178)}
\drawseries{\datafive}{$\ell=5$}{(8,2.9968)}
\drawseries{\dataten}{$\ell=10$}{(8,2.8018)}

\end{tikzpicture}

&
\begin{tikzpicture}[scale=0.5, xscale=0.55, yscale=1.6, font=\tiny]

\def\dataone{(
 1, 2.0547) ( 2, 2.4988) ( 3, 2.8308) ( 4, 3.1055) ( 5, 3.3439) ( 6, 3.5569) ( 7, 3.7508) ( 8 ,3.9297
)}\def\datatwo{(
 1, 1.7194) ( 2, 2.0343) ( 3, 2.2709) ( 4, 2.4671) ( 5, 2.6376) ( 6, 2.79) ( 7, 2.9289) ( 8 ,3.0571
)}\def\datathree{(
 1, 1.5692) ( 2, 1.8248) ( 3, 2.0174) ( 4, 2.1776) ( 5, 2.317) ( 6, 2.4418) ( 7, 2.5555) ( 8 ,2.6605
)}\def\datafour{(
 1, 1.4798) ( 2, 1.6992) ( 3, 1.8652) ( 4, 2.0035) ( 5, 2.1241) ( 6, 2.2321) ( 7, 2.3306) ( 8, 2.4217
)}
\def\datafive{(
1,1.4190) (2,1.6134) (3,1.7612) (4,1.8844) (5,1.9919) (6,2.0886) (7,2.1769) (8,2.2582
)}
\def\dataten{(
1,1.2706) (2,1.4026) (3,1.5043) (4,1.5899) (5,1.6651) (6,1.7328) (7,1.7948) (8,1.8524
 )}

\draw[->] (0,0) -- (\xub+\xbuf,0);
\draw[->] (0,0) -- (0,\yub+\ybuf);

\foreach \x in {\xlb,1,...,\xub}{
    \ifthenelse{\NOT 0 = \x}{
        \draw[thick] (\x,-2pt) -- (\x,2pt);
        \node[anchor=north] at (\x,0) {$\x$};
    }{}
}

\foreach \y in {1,2,3,4,5}{
    \draw[thick] (-3pt,\y) -- (3pt,\y);
    \node[anchor=east] at (0,\y) {$\y$};
}

\node[anchor=north east] at (0,0) {$0$};
\node[anchor=west] at (\xub+\xbuf,0) {$d$};
\node[anchor=south] at (0,\yub+\ybuf) {\footnotesize $ \hat{g}_{d,\ell}^{cm}$};

\drawseries{\dataone}{$\ell=1$}{(8,3.9297)}
\drawseries{\datatwo}{$\ell=2$}{(8,3.0571)}
\drawseries{\datathree}{$\ell=3$}{(8,2.6605)}
\drawseries{\datafive}{$\ell=5$}{(8,2.2582)}
\drawseries{\dataten}{$\ell=10$}{(8,1.8524)}

\end{tikzpicture}
\\
$b_{d,\ell}(n)$ &
$b_{d,\ell}^c(n)$ &
$b_{d,\ell}^m(n)$ &
$b_{d,\ell}^{cm}(n)$ 
\end{tabular}
\caption{Comparison of the exponential growth rates of the four length-graded sequence families for $\ell=1,2,3,5,10$. In the first two panels the values are exact and are determined by Proposition~\ref{prop:sec5:gdl}; in the last two panels the plotted quantities $\hat g_{d,\ell}^m$ and $\hat g_{d,\ell}^{cm}$ are numerical estimates.}
\label{fig:gdl}
\end{figure}
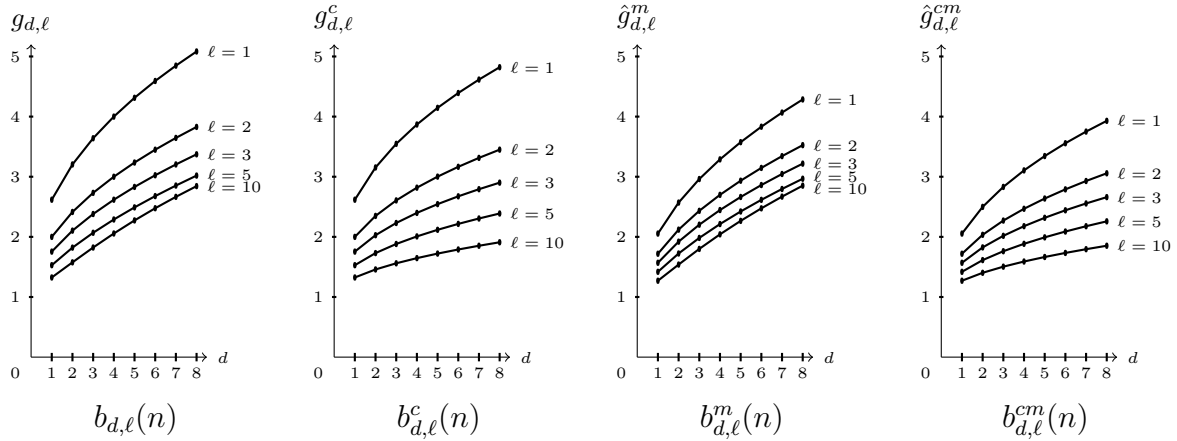

Taken together, the results of Sections~\ref{sec2}--\ref{sec5} show that the same
basic free-operated construction leads to four distinct enumerative regimes,
depending on whether the unary operators and the associative multiplication commute.
In the noncommutative-multiplication settings, this yields explicit algebraic
generating functions, Narayana-type formulas, and lattice-path and tree models;
when multiplication is commutative, the sequence-of-atoms decomposition is replaced
by a multiset-of-atoms decomposition, leading instead to Euler-transform recurrences
and rooted-tree analogs. The asymptotic comparison in this section highlights the
different growth behavior of these four regimes. In particular, this framework both
recovers several classical sequences in the one-operator case and organizes a broad
family of higher-parameter analogs within a common combinatorial setting, several of
which were submitted to the OEIS as part of this work.

\section{Acknowledgment}

We thank Neil Sloane and the many contributors to the On-Line Encyclopedia of Integer Sequences (OEIS) for developing and curating this valuable resource for mathematical research.

\bigskip
\hrule
\bigskip

\noindent
2020 \emph{Mathematics Subject Classification}:
Primary 05A15; Secondary 05A19, 05C05, 18M65.

\noindent
\emph{Keywords}: multi-operator monomial, unary operator, 
Narayana number, Catalan number, Motzkin path, Schr\"oder path, rooted tree, Euler transform, nonsymmetric operad.

\bigskip
\hrule
\bigskip
\noindent
(Concerned with sequences \seqnum{A000012}, \seqnum{A000081}, \seqnum{A000108}, \seqnum{A001003}, \seqnum{A001263}, \seqnum{A002620}, \\ \seqnum{A004148}, \seqnum{A007564}, \seqnum{A023427}, \seqnum{A023432}, \seqnum{A055277}, \seqnum{A055278}, \seqnum{A055279}, \seqnum{A055280}, \seqnum{A059231}, \seqnum{A159771}, \seqnum{A212364}, \seqnum{A239204}, \seqnum{A329691}, \seqnum{A394939}, \seqnum{A394940}, \seqnum{A394941}, \seqnum{A394942}, \seqnum{A394943}, \seqnum{A394944}, \seqnum{A394945}, \seqnum{A394946}, \seqnum{A394947}, \seqnum{A394948}, \seqnum{A394949}, \seqnum{A394950}, \seqnum{A394951}, \seqnum{A394952}, and \seqnum{A394953}.)

\bigskip
\hrule
\bigskip

\end{document}